\documentclass[a4paper,10pt]{amsart}

\usepackage[american]{babel} 
\usepackage[latin1]{inputenc}        
\usepackage{amsmath}
\usepackage{amsthm}
\usepackage{amssymb}
\usepackage{mathtools}
\usepackage{color,graphicx}  
\usepackage{enumitem}   
\usepackage{pstricks-add}
\usepackage{todonotes}
\usepackage{pgf,tikz,pgfplots}
\usetikzlibrary{patterns}
\usetikzlibrary{intersections}
\usepackage{bbm}
\usetikzlibrary{svg.path}
\usetikzlibrary{decorations.pathmorphing}
\usetikzlibrary{hobby,decorations.markings}
\usepackage{mathrsfs}
\usepackage{ifthen}
\usepackage{geometry}
\usepackage{xcolor}
\usepackage{physics}
\usetikzlibrary{svg.path}
\usetikzlibrary{hobby,decorations.markings}

\newtheorem{theo}{Theorem}[section]
\newtheorem{lemma}{Lemma}[section]

\newtheorem{remark}{Remark}[section]

\pgfplotsset{compat=1.16}

\begin{document}

\title{Phase transition for percolation on a randomly stretched square lattice}

\author[Hil\'ario]{Marcelo~R.~Hil\'ario}\address{Universidade Federal de Minas Gerais}\email{mhilario@mat.ufmg.br\\rsanchis@mat.ufmg.br}
\author[S\'a]{Marcos~S\'a}\address{Instituto Nacional de Matem\'atica Pura e Aplicada}\email{marcospy6@ufmg.br\\augusto@impa.br}
\author[Sanchis]{R\'emy~Sanchis}
\author[Teixeira]{Augusto~Teixeira}
\date{}

\begin{abstract}
Let $\{\xi_i\}_{i \geq 1}$ be a sequence of i.i.d.\ positive random variables.
Starting from the usual square lattice replace each horizontal edge that links a site in $i$-th vertical column to another in the $(i+1)$-th vertical column by an edge having length $\xi_i$.
Then declare independently each edge $e$ in the resulting lattice open with probability $p_e=p^\abs{e}$ where $p\in[0,1]$ and $\abs{e}$ is the length of $e$.
We relate the occurrence of nontrivial phase transition for this model to moment properties of $\xi_1$. 
More precisely, we prove that the model undergoes a nontrivial phase transition when $\mathbb{E}(\xi_1^\eta)<\infty$, for some $\eta>1$ whereas, when $\mathbb{E}(\xi_1^\eta)=\infty$ for some $\eta<1$, no phase transition occurs.
\end{abstract}

\maketitle

\section{Introduction}
\label{s:introduction}

In this paper we discuss the existence of phase transition for a percolation model defined on a stretched version of the square lattice where horizontal edges are assigned random lengths while vertical edges remain with length one.
\subsection{Setting and main results}

We start by describing the lattice.
Given an increasing sequence $\Lambda=\{x_0, x_1, \cdots\}\subseteq \mathbb{R}$ called environment define the graph $\mathcal{L}_\Lambda=\big(V(\mathcal{L}_\Lambda\big), E\big(\mathcal{L}_\Lambda)\big)$ whose vertex set $V(\mathcal{L}_\Lambda)$ and edge set $E(\mathcal{L}_\Lambda)$ are given by
\begin{eqnarray}
V(\mathcal{L}_\Lambda)&:=&\Lambda\times\mathbb{Z}_+=\big\{(x,y)\in \mathbb{R}^2;\, x\in\Lambda, y\in \mathbb{Z}_+\big\};
\nonumber\\
E(\mathcal{L}_\Lambda)&:=&\big\{\{(x_i, n), (x_j, m)\}\subseteq V(\mathcal{L}_\Lambda);\, \abs{i-j}+\abs{n-m}=1\big\}\nonumber.
\end{eqnarray} 
Roughly speaking, $\mathcal{L}_\Lambda$ is the lattice obtained from $\mathbb{Z}_+^2$ by stretching or contracting horizontal edges in such a way that the edges connecting sites in $i$-th and $(i+1)$-th column have the same length which is given by the respective difference of consecutive elements in $\Lambda$, namely $x_{i+1}-x_i$.
See Figure \ref{fig3_1}.

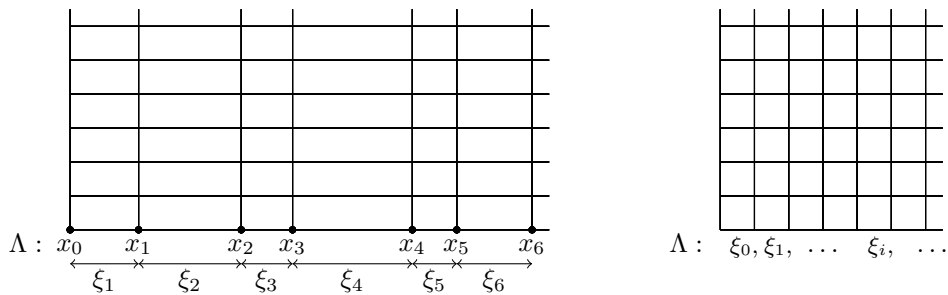
\begin{figure}[htb!]
	\centering
	\begin{tikzpicture}[scale=0.45]
	\foreach \y in {0,1,2,3,4,5,6} {
		\foreach\x in {0,2,5,6.5,10,11.3,13.5}	{
			\draw[ ] (\x,0) to (\x, 6.5);
			\draw[ ] (0,\y) to (14, \y);
			\fill (\x,0) circle(3pt);
	}}
	\def\numbers{{0,2,5,6.5,10,11.3,13.5}}
	   \foreach \i in {1,...,6}{
	    \pgfmathsetmacro{\n}{\numbers[\i-1]}
	     \pgfmathsetmacro{\m}{\numbers[\i]}
	     \draw[<->] (\n,-1) -- (\m,-1);
	     \node at (\n*.5+\m*.5,-1.5) {$\xi_{\i}$};
	     
	     }
	\node at (0,-0.5) {$x_0$};
	\node at (2,-0.5) {$x_1$};
	\node at (5,-0.5) {$x_2$};
	\node at (6.5,-0.5) {$x_3$};
	\node at (10,-0.5) {$x_4$};
	\node at (11.3,-0.5) {$x_5$};
	\node at (13.5,-0.5) {$x_6$};
	\node at (-1.3,-0.4) {$\Lambda:$};
	
	\foreach \y in {0,1,2,3,4,5,6} {
		\foreach\x in {0,1,2,3,4,5,6}	{
			\draw[ ] (\x+19,0) to (\x+19, 6.5);
			\draw[ ] (0+19,\y) to (6.5+19, \y);
	}}
	\node at (19.7,-0.5) {\small{${\xi_0},$}};
	\node at (20.7,-0.5) {\small{${\xi_1},$}};
	\node at (22,-0.6) {$\ldots$};
	\node at (23.7,-0.5) {\small{${\xi_i},$}};
	\node at (25.2,-0.6) {$\ldots$};
	\node at (18,-0.4) {$\Lambda:$};
	\end{tikzpicture}
	\caption{We illustrate the lattice $\mathcal{L}_\Lambda$ (on the left) and the alternative formulation on $\mathbb{Z}^2_+$ (on the right).
	The environment $\Lambda$ can be specified either by the $x_i$'s (right) or by the $\xi_i=x_i-x_{i-1}$ (left).
}\label{fig3_1}
\end{figure}
We wish to allow for random environments $\Lambda$.
For that we will assume that $\Lambda$ is distributed according to a renewal process as we describe next.
Let $\xi$ be a positive random variable and $\{\xi_i\}_{i\in\mathbb{Z}_+^*}$ a sequence composed of i.i.d.\ copies of $\xi$.
Set 
\begin{equation}
\label{e:def_lambda}
\Lambda:=\bigg\{\sum_{1\leq i\leq k} \xi_i; \ k\in\mathbb{Z}_+\bigg\}=\big\{x_k\in \mathbb{R};\,  x_0=0 \textrm{ and }  x_k=x_{k-1}+\xi_k\textrm{ for }k\in\mathbb{Z}_+^* \big\}.
\end{equation}
Thus $\xi_i$ gives the random separation between the $i$-th and $(i+1)$-th vertical columns in the stretched lattice.
The resulting sequence $\Lambda$ is called a renewal process with interarrival distribution $\xi$.
We denote $\upsilon_\xi(\cdot)$ the law of this renewal process. 
An overview on renewal processes will be provided in Section \ref{s:ren_pro}.

Given a realization of the environment $\Lambda$ we can define a bond percolation process in $\mathcal{L}_\Lambda$.
For each $p\in[0,1]$, denote $\mathbb{P}^\Lambda_p(\cdot)$ the probability measure on $\{0,1\}^{E(\mathcal{L}_\Lambda)}$ under which the random variables $\{\omega(e)\}_{e\in E(\mathcal{L}_\Lambda)}$ are independent Bernoulli random variables with mean 
\begin{equation}
\label{eq:_p_e}
p_e=p^\abs{e}
\end{equation}
where, for each edge $e:=\{v_1, v_2\}\in  E(\mathcal{L}_\Lambda)$, $\abs{e}=\norm{v_1-v_2}$ denotes the Euclidean length of $e$.

We write $o\leftrightarrow \infty$ for the event that there exists an infinite path starting at $o=(0,0)$ that only uses open edges, that is edges $e$ for which $\omega(e)=1$.

Our main results relate moment properties of $\xi$ to whether or not the resulting percolation process exhibits a non-trivial phase transition.
\begin{theo}\label{t3_1} Let $\xi$ be a positive random variable with $\mathbb{E}(\xi^\eta)<\infty$ for some $\eta>1$. 
Then there exists $p_c \in (0,1)$, depending on the law of $\xi$ only, such that for $p<p_c$
    \[
    \mathbb{P}^\Lambda_p(o\leftrightarrow\infty)=0, \textrm{ for }\upsilon_\xi\textrm{-almost all }\Lambda
    \]
    whereas for $p>p_c$,
    \[
    \mathbb{P}^\Lambda_p(o\leftrightarrow\infty)>0, \textrm{ for }\upsilon_\xi\textrm{-almost all }\Lambda.
    \]
    \end{theo}
    \begin{theo}\label{t3_2} Let $\xi$ be a positive random variable with $\mathbb{E}(\xi^\eta)=\infty$ for some $\eta<1$.
    Then for any $p\in[0,1)$,  
    \[
    \mathbb{P}^\Lambda_p(o\leftrightarrow\infty)=0, \textrm{ for }\upsilon_\xi\textrm{-almost every environment }\Lambda.
    \]
    \end{theo}

Theorem \ref{t3_1} states that the model undergoes a non-trivial phase transition when the renewal increments are heavy-tailed as long as they have finite moments of any order greater than one.
Its proof relies on the control of the environment and of crossing events in the resulting stretched lattice via multiscale analysis.
The control of the crossing events is reminiscent of the one presented in \cite{Bramson91} where the authors study the survival of a contact process in a random environment.
However, the techniques developed there, when translated to our context, seem only to apply when the $\xi$'s have geometric distribution.

Theorem \ref{t3_2} rules out the occurrence of a nontrivial phase transition when the increments of the renewal process have sufficiently heavy tails. 
Not very surprisingly, this phenomenon stems from the fact that the consecutive columns will be typically located very far apart. 
Its proof is presented in Section \ref{demt3_2} and consists of a Borel-Cantelli argument.

An interesting problem is to determine whether phase  transition occurs when the lattice is stretched both horizontally and vertically. 
As shown in \cite{Hoffman05} the answer is positive when $\xi$ has geometric distribution.
For dimensions $d \geq 3$, non-trivial phase transition was shown to take place in a very similar setup where the edges are stretched according to exponential random variables in \cite{Jonasson00}.
We are currently unable to tackle the problem when the stretching is made according to general renewal processes in more than one direction.

One may ask whether the moment conditions appearing in Theorems \ref{t3_1} and \ref{t3_2} could be relaxed.
This may be done for Theorem \ref{t3_2}, for instance by imposing only that $\mathbb{E}(\xi \log({1+\xi})^{-\varepsilon})=\infty$ for some $\varepsilon>0$, without complicating too much its proof.
However, for Theorem \ref{t3_1}  it is not obvious how to get a substantial improvement.
Indeed, the assumed moment condition is crucial to obtain a good decoupling bound for the environment.
A weaker condition would result in a worse control of the environment that can be overcome by changing the scale progression in which we analyse the system. 
However, such change would result in several difficulties in the control of the crossing events.
As we explain better in Remark \ref{r:decoupling_bound}, our results can be extended beyond the case of $i.i.d.$ $\xi$'s provided that the environment satisfies a similar decoupling bound.

The interesting question whether the dichotomy for occurrence of non-trivial phase transition is determined by the existence or not of the first moment of $\xi$ is currently out of reach of our techniques and seems to be a complicated question.

\subsection{Related work and motivation}

Our main motivation in this work is to study how the presence of impurities  may affect the phase transition in a given system.
This is an important subject in disordered systems and has been investigated in a number of contexts including percolation, polymer models, spin models and interacting particle systems.

By the word impurities, we mean either defects in the structure of the underlying lattice such as dilution, or inhomogeneities in the parameters that govern local interactions such as the probability of opening the edges.

In the context of finite weighted graphs, with randomness in the weights, Aldous \cite{aldous16} studied the emergence of the giant component.

For the model we consider here the disorder in the model has a twofold interpretation, as explained in Section \ref{s:equiv_form}.
Indeed, $\eqref{ipmsl}$ defines a model on $\mathbb{Z}^2_+$ with inhomogeneous probabilities of connecting neighboring sites while the formulation following from \eqref{e:integer_stretched} can be interpreted as a homogeneous percolation with parameter $p$ on a dilute lattice.

In the context of percolation the problem of strictly inequalities for the critical threshold when dilution is made regularly and deterministically was studied in  \cite{Menshikov87}.
A great step forward in the study of strict inequalities was the method of differential inequalities in \cite{Aizenman91}.
For instance, it was employed in \cite{Chayes00} to study the so-called mixed percolation that can be regarded as site percolation on a dilute  lattice where egdes are removed in an i.i.d.\ fashion.
There, the authors prove continuity and strict monotonicity of the threshold as a function of the dilution density.

For ferromagnetic spin systems, the effect of i.i.d.\ lattice dilution on the phase transition has been rigorously studied (see for instance, \cite{Griffiths68, Georgii81, Georgii84, Aizenman87_2}).

In the present paper, rather than in an i.i.d.\ fashion, disorder is introduced according to quenched realizations of subsets of the lattice that extend infinitely far in a single direction, namely the set of vertical columns in $\Lambda \times \mathbb{R}_+$.
We then say that the model exhibits columnar disorder.

Columnar disorder for the two-dimensional Ising model was introduced in the pioneering work of McCoy and Wu \cite{McCoy68}. 
There the coupling constants between neighboring spins are random and vary depending on the vertical columns to which they belong much in the spirit of \eqref{ipmsl}.
Back to percolation models, using the notation in \eqref{e:integer_stretched} other models in which the parameters are given by $p_e= p \mathbf{1}_{\{e\in E_{\text{vert}}(\Lambda)\}}+ q\mathbf{1}_{\{e\not\in E_{\text{vert}}(\Lambda)\}}$
have been studied.
It is clear that the behavior of the system will depend on the choices of $\Lambda$, $p$ and $q$.
However, in contrast to the situation in \cite{Chayes00} where a trade-off between the dilution and interaction strength can be obtained in order to control the critical curve, in the presence of columnar disorder, the picture is much more modest.
In fact, even determining whether a non trivial phase transition persists or proving strictly inequalities for the critical threshold may be a hard question.
For instance, in \cite{Zhang94}, it is shown that, in dimension $2$, enhancement along a single line of defects (that is, $\Lambda = \{0\}$) does not alter the critical threshold.
In \cite{Copin18} the behaviour of the phase transition in the presence of an i.d.d.\ mixture of enhanced and weaken vertical columns was studied.
The persistence of percolation for a directed percolation model propagating across a set of rare weaken transverse columns was studied in  \cite{Kesten12}.

Regarding time as an extra dimension, the same type of disorder arises also in the study of one-dimensional interacting particle systems in random environment, for instance the voter model \cite{Ferreira90} and the contact process \cite{Bramson91, Liggett92, Madras94, Newman96}.
As mentioned before, we believe that the methods of \cite{Bramson91} can be adapted to prove the existence of percolation in our setting when $\xi$ has geometric distribution.

For the contact process in higher dimensional random environment, \cite{Klein94} establishes moment conditions for non-survival in a setting somewhat similar to the one in \cite{Campanino91} where $d$-dimensional Ising model with lower-dimensional disorder was studied.

For percolation on stretched Euclidean lattices, the existence of phase transition was obtained in dimensions $d\geq 3$ when the lattice is stretched in all the three directions according to exponential random variables \cite{Jonasson00}.
A few years later, in \cite{Hoffman05} the $d=2$ case was also settled.

\subsection{Equivalent formulations}
\label{s:equiv_form}

Consider the first quadrant of the square lattice as $\mathbb{Z}^2_+=(V(\mathbb{Z}^2_+), E(\mathbb{Z}^2_+))$ where 
 \begin{eqnarray}
 V(\mathbb{Z}^2_+)&=&\big\{v=(v_1,v_2)\in \mathbb{R}^2;\, v_1, v_2\in\mathbb{Z}_+\big\},\nonumber\\
 E(\mathbb{Z}^2_+)&=&\big\{\{v,w\}\subseteq V(\mathbb{Z}^2_+);\, \abs{v_1-w_1}+\abs{v_2-w_2}=1\big\},\nonumber
 \end{eqnarray}
and where we write $\mathbb{Z}_+ = \{0,1,2\ldots\}$.
We often abuse notation and do not distinguish between $V(\mathbb{Z}^2_+)$ and $\mathbb{Z}^2_+$, and similarly for other graphs.

We write $v\sim w$ if $v$ is a {\it neighbor} of $w$, i.e. $\{v, w\}\in E(\mathbb{Z}^2_+)$. 
A {\it path} in $A\subseteq \mathbb{Z}^2_+$ is a sequence of sites $v_0\sim v_1\sim \cdots\sim v_n$ such that $v_i\in A$ for~all~$i$.
Given a bond percolation configuration $\omega\in\{0,1\}^{E(\mathbb{Z}^2_+)}$, an edge $e\in E(\mathbb{Z}^2_+)$ is said {\it open} if $\omega(e)=1$, otherwise it is said {\it closed}. 
For two sites $v, w\in \mathbb{Z}_+^2$, $v$ and $w$ are {\it connected} (denoted $v\leftrightarrow w$) if there exists a sequence $v=v_0\sim v_1\sim \cdots \sim v_n=w$ such that $\omega(\{v_i, v_{i+1}\})=1$ for every $0\leq i<n$.
The cluster of a site $v$ is the set of all sites $w$ such that $v\leftrightarrow w$, and we denote by $\{v\leftrightarrow \infty\}$ the event where the cluster of $v$ has infinite cardinality.

The percolation model on $\mathcal{L}_\Lambda$ defined by \eqref{eq:_p_e} may be mapped into the percolation on $\mathbb{Z}^2_+$ where, conditional on $\xi_1, \xi_2, \ldots$, each edge $e\in E(\mathbb{Z}^2_+)$ is declared open independently with probability 
\begin{eqnarray}
p_e=\left\{
\begin{array}{ll}
p,& \text{if $e=\{(i,j),(i, j+1)\}$ for some $i,j$,}\\
p^{\xi_{i+1}},& \text{if $e=\{(i,j),(i+1, j)\}$ for some $i,j$.}
\end{array}
\right.
\label{ipmsl}
\end{eqnarray}
In this alternative formulation, conditional on $\xi_1, \xi_2, \ldots$, the resulting bond percolation process in $\mathbb{Z}^2_+$ is inhomogeneous, unless the distribution of $\xi$ is concentrated on $1$.
If we average in the realization of the $\xi_i$'s we obtain a model that is homogeneous but with infinite-range dependencies along the vertical direction.

In the case where $\xi$ is positive and integer-valued we can map the percolation model defined on $\mathcal{L}_\Lambda$ (with $\Lambda\subseteq\mathbb{Z}_+$ defined as in \eqref{e:def_lambda}) with parameters given by \eqref{eq:_p_e} to yet another equivalent model on $\mathbb{Z}^2_+$ as follows.
Let
\[
E_{\text{vert}}(\Lambda^c)\vcentcolon=\big\{\{(x, y), (x, y+1)\}\in E(\mathbb{Z}_+^2); x\not\in \Lambda, y\in \mathbb{Z}_+\big\}.
\] 
Let each edge $e\in E(\mathbb{Z}^2_+)$ be open independently with probability 
\begin{equation}
\label{e:integer_stretched}
p_e=\left\{
\begin{array}{ccc}
0& \textrm{ if }e\in E_{\text{vert}}(\Lambda^c),\\
p& \textrm{ if }e\not\in E_{\text{vert}}(\Lambda^c),
\end{array}\right.
\end{equation}
and closed otherwise.

Geometrically, this formulation can also be viewed as a bond percolation model in a dilute lattice obtained from $\mathbb{Z}^2_+$ by removing the edges lying in vertical columns that project to $\Lambda^c$ while preserving all other edges.
The resulting graph is similar to the stretched lattice $\mathcal{L}_\Lambda$ above, only that now the edges are split into unit length segments.
Each one of these is open independently with probability $p$.
One can recover the original formulation on $\mathcal{L}_\Lambda$ by declaring an edge open if all the corresponding unit length edges in $\mathbb{Z}^2_+$ are open in the new formulation.

For convenience, we will use the different formulations of the model in different parts of the text. 
Since they are all equivalent, abuse notation denoting also $\mathbb{P}^\Lambda_p(\cdot)$ the law of the new versions and hope it will be clear from the context which formulation we are using.

\subsection{Overview of the paper}

We now present an brief overview of the paper.
Section \ref{s:ren_pro} contains a brief review on renewal processes and the proof a decoupling inequality which is crucial to the control of the environment \eqref{t3_1}.
In Section \ref{s:multiscale}, we develop the multiscale scheme that will be used to control the environment and the percolation process in order to obtain the proof of Theorem \ref{t3_1}.
First, in Section \ref{multiscale1} we define a fast-growing sequence of numbers which correspond to the scales in which we analyze the model. 
Then we partition $\mathbb{Z}_+$ into the so-called blocks which are intervals whose lengths are related to the scales.
A block at scale $k$ will be labeled either bad or good hierarchically depending on whether or not the renewal process within it has arrivals that are close to each other, respectively.
We will show that bad blocks are extremely rare. 
Section \ref{s:multiscale2} is destined to the construction of crossing events in rectangles that extend very far vertically and whose basis project into blocks. 
We show that said crossings occur with very high probability if the projection fall into good blocks.
In Section~\ref{demt3_1}, we finish the proof of Theorem \ref{t3_1} putting together the fact that the blocks are most likely good and that good blocks are easy to cross.
Theorem \ref{t3_2} is proved in Section \ref{demt3_2}.

\subsection*{Acnowledgements}

The research of AT was partially supported by CNPq grants `Produtividade em Pesquisa' (304437/2018-2) and `Projeto Universal' (304437/2018-2), and by FAPERJ grant (202.716/2018).
The research of MH was partially supported by CNPq grants `Projeto Universal' (406659/2016-8) and `Produtividade em Pesquisa' (307880/2017-6) and by FAPEMIG grant `Projeto Universal' (APQ-02971-17)MS was supported by CAPES.
RS was supported by CNPq (grant 310392/2017-9), CAPES and FAPEMIG (PPM 0600/16).

\section{Renewal processes}
\label{s:ren_pro}

The purpose of this section is to prove a decoupling inequality (Lemma \ref{dec_ine}), which will be used as a fundamental tool in our multiscale analysis in Section \ref{multiscale1}. 
We start presenting a brief outline of some results on renewal processes.

Let $\xi$ and $\chi$ be integer-valued random variables called {\it interarrival time} and {\it delay}, respectively.
We assume that $\xi \geq 1$ and $\chi \geq 0$ a.s.
Let $\{\xi_i\}_{i\in\mathbb{Z}^*_+}$ be i.i.d.\ copies of $\xi$, also independent of $\chi$.
We define the {\it renewal process}
$$X=X(\xi,\chi)=\{X_i\}_{i\in\mathbb{Z}_+}$$
recursively as:
$$X_0=\chi,  \textrm{\ \ \ \ and \ \ \ \ }X_i=X_{i-1}+\xi_i \textrm{ for }i\in\mathbb{Z}_+^*.$$
We say that the $i$-th renewal occurs at time $t$ if $X_{i-1}=t$.
The law of $X$ regarded as a random element on a probability space supporting $\chi$ and the i.i.d.\ copies of $\xi$ will be denoted by $\upsilon_\xi^\chi$.

It is convenient to define two other processes 
\[
Y=Y(\xi,\chi)=\{Y_n\}_{n\in\mathbb{Z}_+}\textrm{\ \ \ \ and \ \ \ \ }Z=Z(\xi,\chi)=\{Z_n\}_{n\in\mathbb{Z}_+},
\]
as
\begin{eqnarray}
Y_n=\left\{
\begin{array}{lll}
1,& \textrm{ if a renewal of }X\textrm{ occurs at time } n,\\
0,&\textrm{ otherwise}.
\end{array}\right.
\label{defY}\end{eqnarray}
and
\begin{eqnarray}
Z_n=\min\{X_i-n; \ i\in\mathbb{Z}_+\textrm{ and } X_i-n\geq 0\}.\label{defZ}
\end{eqnarray}
Since each one of the processes  $X$, $Y$ and $Z$ fully determines the two others (see Figure \ref{figurexyz}), $Y$ and $Z$ will also be called renewal processes with interarrival time $\xi$ and delay $\chi$. 
We abuse notation and write $\upsilon_\xi^\chi$ for the law of $Y$ and $Z$. It is worth noticing that $Z$ is a Markov chain.

\begin{figure}[htb!]
	\centering
	\begin{tikzpicture}[scale=0.8, every node/.style={scale=0.8}]
	\draw[thick,->](-1,1.5)--(12.7,1.5);
	\foreach\x in {-1,0,1,...,12}\draw[thick] (\x,1.6)--(\x,1.4);
	\draw[thick,->](-1,0.3)--(12.7,0.3);
	\foreach\x in {-1,0,1,...,12}\draw[thick] (\x,0.4)--(\x,0.2);
	\draw[thick,->](-1,-0.9)--(12.7,-0.9);
	\foreach\x in {-1,0,1,...,12}\draw[thick] (\x,-1)--(\x,-0.8);
	\foreach\x in {0,2,7,10,11}\filldraw (\x,1.5) circle(3pt);
	\foreach\x/\y in {0/0,2/1,7/2,10/3,11/4}\node at (\x,1.1) {$X_{\y}$};
	\foreach\x in {0,2,7,10,11}\node at (\x,-0.1) {$1$};
	\foreach\x in {-1,1,3,4,5,6,8,9,12}\node at (\x,-0.1) {$0$};
	\foreach\x/\y in {-1/1,0/0,1/1,2/0,3/4,4/3,5/2,6/1,7/0,8/2,9/1,10/0,11/0,12/\cdots}\node at (\x,-1.3) {$\y$};
	\draw[->] (0.2,1.7) arc (90+43.6:90-43.6:1.16);
	\draw[->] (2.2,1.7) arc (90+18.59:90-18.59:6.9);
	\draw[->] (7.2,1.7) arc (90+30.96:90-30.96:2.53);
	\draw[->] (10.2,1.8) arc (90+65.8:90-65.8:0.3);
	\draw (11.2,1.7) arc (90+43.6:80:1.3);
	\draw[->] (-0.8,1.7) arc (90+65.8:90-65.8:0.3);
	\foreach\x/\y in {1/1,4.5/2,8.5/3,10.5/4,12/5}\node at (\x,2.3) {$\xi_{\y}$};
	\node at (-0.5,2.3) {$\chi$};
	\node at (-1,1.1) {0};
	\foreach\x/\y in {1.1/X, -0.1/Y_n, -1.3/Z_n}\node at (-1.9,\x+0.2) {$\y:$};
	\end{tikzpicture}
	\caption{Illustration of the processes $X,Y,Z$. In this realization, $\chi=1$, $\xi_1=2$, $\xi_2=5$, $\xi_3=3$ and $\xi_4=1$. \label{figurexyz}}
\end{figure}
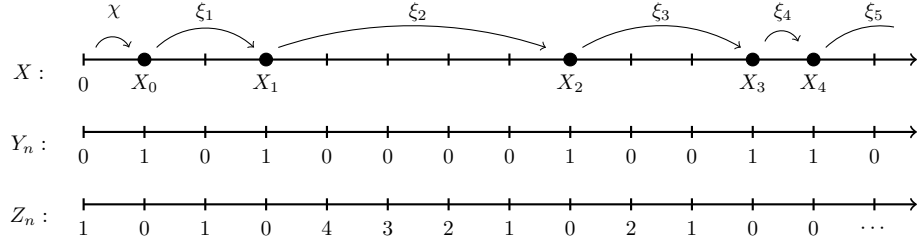

For $m\in\mathbb{Z}_+$ consider $\theta_m:\mathbb{Z}^\infty\mapsto \mathbb{Z}^\infty$, the shift operator given by 
\[
\theta_m(x_0, x_1, \cdots)=(x_m, x_{m+1},\cdots).
\]
It is desirable that $Z$ be invariant under shifts, that is,
\begin{eqnarray}\label{thetam}
\theta_m Z\overset{d}{=} Z\textrm{ \ \ for any }m\in\mathbb{Z}_+^*.
\end{eqnarray}

When $\mathbb{E}(\xi)< \infty$, we can define a random variable $\rho=\rho(\xi)$ with distribution
\begin{equation}
\label{e:stat_decay}
\rho_k=\mathbb{P}(\rho=k):=\dfrac{1}{\mathbb{E}(\xi)}\sum_{i=k+1}\mathbb{P}(\xi=i),\ \textrm{for any }k\in\mathbb{Z}_+,
\end{equation}
independent of everything else.
It is straightforward to show that, using $\rho$ as the delay yields a Markov process $Z(\xi,\rho)$ satisfying \eqref{thetam}.
For this reason, the random variable $\rho$ with distribution given by \eqref{e:stat_decay} is called {\it stationary delay}. 
In particular, 
\begin{eqnarray}
Z_n\overset{d}{=}Z_0\overset{d}{=}\rho.
\label{e:Z_rho}
\end{eqnarray} 
Also notice that 
\begin{eqnarray}
\text{if \,\,\,$
\mathbb{E}(\xi^{1+\varepsilon})<\infty,$ \,\,\, then \,\,\,  
$\mathbb{E}(\rho^{\varepsilon})<\infty$}.
\label{finitudederho}
\end{eqnarray}

Let $X=X(\xi,\chi)$ and $X'=X(\xi,\chi')$ be two independent renewal processes with interarrival time $\xi$ and delays $\chi$ and $\chi'$ respectively, and denote $\upsilon^{\chi,\chi'}_\xi(\cdot)$ the product measure $\upsilon^{\chi}_\xi\otimes\upsilon^{\chi'}_\xi$. Recall the definition of $Y$, $Y'$ in \eqref{defY} and define 
\[
T:=\min\{k\in\mathbb{Z}_+^*; \ Y_k=Y_k'=1\},
\]
the {\it coupling time} of $X$ and $X'$.

We say that $\xi$ is {\it aperiodic} if 
\[
\gcd\big\{k\in\mathbb{Z}^*_+; \ \mathbb{P}(\xi=k)>0\big\}=1.
\]

We now proceed to prove a decoupling inequality for stationary renewals. 
In order to do so, we will bound $\upsilon_\xi^{\chi, \chi'}(T>n)$ above applying Markov's inequality to $T^\varepsilon$, where $\varepsilon>0$.
The next theorem whose proof can be found in \cite[Theorem 4.2, p.\ 27]{Lindvall02}, establishes sufficient conditions on the delays and on the interarrival time for $T^\varepsilon$ to have finite expectation.
We write $\mathbb{E}_\xi^{\chi,\chi'}(\cdot)$ for the expectation with respect to  $\upsilon^{\chi,\chi'}_\xi$.

\begin{theo} \cite[Theorem 4.2]{Lindvall02}
\label{lind} Let $\xi$ be an aperiodic positive integer-valued random variable. 
Suppose that for some $\varepsilon\in(0,1)$, $\mathbb{E}(\xi^{1+\varepsilon})<\infty$, and that $\chi$, $\chi'$ are non-negative integer-valued random variables with $\mathbb{E}(\chi^\varepsilon)$ and $\mathbb{E}(\chi'^\varepsilon)$ finite. 
Then $\mathbb{E}^{\chi',\chi'}_\xi(T^\varepsilon)<\infty$.
\end{theo}

We can now prove the desired decoupling inequality for renewal processes:
\begin{lemma} 
\label{dec_ine} 
Let $\xi$ be an aperiodic positive integer-valued random variable with $\mathbb{E}(\xi^{1+\varepsilon})<\infty$, for some $\varepsilon>0$, and consider the renewal process $Y=Y\big(\xi,\rho(\xi)\big)$ defined in (\ref{defY}).
Then there exists $c_1=c_1(\xi,\varepsilon)\in(0,\infty)$ such that for all $n,m\in\mathbb{Z}_+$ and for every pair of events $A$ and $B$, with
\[
A\in\sigma(Y_i;\ 0\leq i\leq m)\ \ \ \textrm{ and } \ \ \ B\in\sigma(Y_i, \ i\geq m+n)
\]
we have
\begin{equation}
\label{e:decoup_upsilon}
\upsilon^\rho_\xi(A\cap B)\leq \upsilon^\rho_\xi(A)\upsilon^\rho_\xi(B)+c_1 n^{-\varepsilon}.
\end{equation}
\end{lemma}

\begin{proof} 
If $\upsilon^\rho_\xi(A)=0$ there is nothing to be proved and we assume 
henceforth that $\upsilon^\rho_\xi(A)>0$. 
Recall the definition of $Z$ in (\ref{defZ}). 
Using Markov's property we have
\begin{eqnarray}
\upsilon^\rho_\xi(A\cap B)&=&\upsilon^\rho_\xi\big(A\cap B\cap\{Z_m>n/2\}\big)+\upsilon^\rho_\xi\big(A\cap B\cap \{Z_m\leq n/2\}\big)\nonumber\\
&\leq&\upsilon^\rho_\xi\big(Z_m>n/2\big)+\upsilon^\rho_\xi(A)\hspace{-0.5cm}\sum_{\substack{0\leq i\leq\lfloor n/2\rfloor;\\ \upsilon^\rho_\xi(Z_m=i|A)>0}}  \hspace{-0.5cm}\upsilon^\rho_\xi\big(B|A, Z_m=i\big)\upsilon^\rho_\xi\big(Z_m=i|A\big)\nonumber\\
&\leq&\upsilon^\rho_\xi\big(Z_m>n/2\big)+\upsilon^\rho_\xi(A)\hspace{-0.1cm}\max_{0\leq j\leq\lfloor n/2\rfloor}\hspace{-0.1cm} \upsilon^{\delta_{m+j}}_\xi(B).
\label{dd}
\end{eqnarray}

Let us now compare $\upsilon_\xi^{\delta_{m+j}}(B)$ with $\upsilon_\xi^{\rho}(B)$, when $0\leq j\leq\lfloor{n}/{2}\rfloor$. 
Using that $\upsilon_\xi^{\delta_{m+j}}(B)=\upsilon_\xi^{\delta_0}\big(\theta_{m+j}(B)\big)$, the stationarity of $\rho$ and a standard coupling for Markov chains we have
\begin{eqnarray}\abs{\upsilon_\xi^{\delta_{m+j}}(B)-\upsilon_\xi^{\rho}(B)}&=&\abs{\upsilon_\xi^{\delta_0}\big(\theta_{m+j}(B)\big)-\upsilon_\xi^{\rho}\big(\theta_{m+j}(B)\big)}\nonumber\\ 
&{\leq}&\upsilon^{\delta_0, \rho}_\xi(T>n-j)\nonumber\\
&\leq&\upsilon^{\delta_0, \rho}_\xi(T>n/2).\label{ddd}\end{eqnarray}
By (\ref{dd}), (\ref{ddd}) and the fact that we have $Z_m\stackrel{d}=Z_0\stackrel{d}{=}\rho$
\begin{eqnarray}\upsilon^\rho_\xi(A\cap B)&\leq& \upsilon^\rho_\xi(A) \upsilon^{\rho}(B)+\upsilon^\rho_\xi(\rho>n/2)+\upsilon^{\delta_0, \rho}_\xi(T>n/2)\nonumber\\
&\leq& \upsilon^\rho_\xi(A) \upsilon^{\rho}(B)+{2^\varepsilon\mathbb{E}(\rho^\varepsilon)}n^{-\varepsilon}+2^\varepsilon\mathbb{E}^{\delta_0, \rho}_\xi(T^\varepsilon)n^{-\varepsilon},
\nonumber
\end{eqnarray}
where $\mathbb{E}(\rho^\varepsilon)$ and $\mathbb{E}^{\delta_0, \rho}_\xi(T^\varepsilon)$ are finite by (\ref{finitudederho}) and Theorem \ref{lind}, respectively.
Defining $c_1={2^\varepsilon\mathbb{E}(\rho^\varepsilon)}+2^\varepsilon\mathbb{E}^{\delta_0, \rho}_\xi(T^\varepsilon)$ concludes the proof.
\end{proof}

\section{The multiscale scheme}
\label{s:multiscale}

This section is divided into two parts.
The first one is dedicated to the control of the environment and the second  to the control of the probability of the occurrence of crossing events in large boxes.
Throughout the whole section we restrict ourselves to the case when $\xi$ is positive, integer-valued and aperiodic.
Moreover we assume that $\mathbb{E}(\xi^{1+\varepsilon})<\infty$ for a some $\varepsilon>0$ and denote $\rho=\rho(\xi)$ the respective stationary delay given in \eqref{e:stat_decay}.
This allows us to apply the results obtained in Section \ref{s:ren_pro} apply.
Since $\xi$ is integer-valued, we only work with the version of the model which is defined on the $\mathbb{Z}^2_+$-lattice and where the probability of opening edges are given by \eqref{e:integer_stretched}.

\subsection{Environments}
\label{multiscale1}
Let us now fix a sequence of lengths $L_0, L_1, L_2, \ldots$, called scales.
In order to choose the initial scale, $L_0$, let us start by fixing constants
\begin{eqnarray}
\label{defalp}\alpha\in\big(0,\tfrac{\varepsilon}{2}\big]\textrm{\ \ \  and \ \ \ } \gamma\in\big(1, 1+\tfrac{\alpha}{\alpha+2}\big).
\end{eqnarray} 
They will appear as exponents in several expressions below.
The exponent $\gamma$ is related to the rate of growth for the sequence of scales in which we will study the environment while $\alpha$ will provide the rate of decay of the probability that bad events (to be defined later) occur in each scale (see (\ref{Lk}) and (\ref{cotlenv})).

In order to define the sequence of scales, let us fix $L_0=L_0(\xi, \varepsilon, \alpha, \gamma)\in\mathbb{Z}_+$ sufficiently large so that
\begin{enumerate}[label=(\roman*)]
\item\label{L1} $L_0^{\gamma-1}\geq 3$,
\item \label{L2}$L_0^{\varepsilon-\alpha}\geq \mathbb{E}(\rho^\varepsilon)$ and
\item \label{L3}$L_0^{c_2}\geq c_1+1$, where $c_1$ is given by Lemma \ref{dec_ine} and $c_2=2+2\alpha-\gamma\alpha-2\gamma$.
(Notice that our bound for $\gamma$ in (\ref{defalp}) ensures that $c_2>0$.)
\end{enumerate}
Once $L_0$ is fixed, we can define recursively the sequence of scales $(L_k)_{k\in\mathbb{Z}_+}$ by 
\begin{eqnarray}\label{Lk}  L_k=L_{k-1}\lfloor L_{k-1}^{\gamma-1}\rfloor, \textrm{ for any }k\geq1.
\end{eqnarray}
Item \ref{L1} in the definition of $L_0$ together with (\ref{defalp}) and (\ref{Lk}) implies that these scales grow super-exponentially fast. 
In fact,
\begin{eqnarray}\label{cotaL}\left(\frac{2}{3}\right)^k L_{0}^{\gamma^k}\leq\cdots\leq\frac{2}{3}L_{k-1}^\gamma\leq L_k\leq L_{k-1}^\gamma\leq\cdots\leq L_0^{\gamma^k}.\end{eqnarray}
Items \ref{L2} and \ref{L3} are technical and will be used  to prove Lemma \ref{lmult1} below.

For $k\in\mathbb{Z}_+$, let us partition $\mathbb{R}_+$ into intervals of length $L_k$
\[
I_j^k:=\big[j L_k, (j+1)L_k\big), \textrm{ with } j\in\mathbb{Z}_+.
\]
The interval $I_j^k$ is called the $j$-th block at scale $k$. 
For $k \geq 1$, each block at scale $k$ can be partitioned into disjoint blocks at scale $k-1$:
\begin{equation}\label{firstandlast} I_j^k=\bigcup_{i \in l_{k,j}} I_i^{k-1},
\end{equation}
where the union runs over the set of indices
\begin{equation}
 l_{k,j} := \{ i \in \mathbb{Z}_+;\, I^{k-1}_i \cap I^k_j \neq \varnothing \} = \big\{ j\lfloor L_{k-1}^{\gamma-1}\rfloor, \cdots, (j+1)\lfloor L_{k-1}^{\gamma-1}\rfloor-1 \big\}
 \end{equation}
whose cardinality satisfies
\begin{equation}
\label{e:|lkj|}
    |l_{k,j}| = \lfloor L_{k-1}^{\gamma-1}\rfloor.
\end{equation}

Now fix an environment $\Lambda\subseteq \mathbb{Z}_+$. 
Blocks will be labeled either good or bad according to $\Lambda$ recursively as follows.
For $k=0$ declare the $j$-th block at scale $0$, $I_j^0$, good if $\Lambda\cap I_j^0\not=\varnothing$, and bad otherwise. 
Once the blocks at scale $k-1$ are all labeled we declare a block at scale $k$ bad if it contains at least two non-consecutive bad blocks at scale $k-1$, and good otherwise.
More precisely, for $j, k \in\mathbb{Z}_+$ consider the events $A^k_j$ defined recursively  by 
\begin{eqnarray} 
A^0_j&=&\{\Lambda\subseteq{\mathbb{Z}_+}; \Lambda\cap I_j^0=\varnothing\} \nonumber\\
A^k_j&=&\hspace{-0.5cm}\bigcup_{\substack{ i_1, i_2 \in l_{k,j} \\ \abs{ i_1-i_2 } \geq  2}}\hspace{-0.3cm} \left(A_{i_1}^{k-1}\cap A_{i_2}^{k-1}\right), \textrm{ for }k\geq 1.
\label{AA}
\end{eqnarray}
Sometimes we will write $\{I_j^k\textrm{ is bad}\}$ instead of $A_j^k$.
We will also write $\{I_j^k\textrm{ is good}\}$ for the complementary set of environments.
By \ref{L1}, $l_{k,j}$ has at least three elements, so that the union in \eqref{AA} always runs over nonempty collections of indices.

We now define
\[
p_k:=\upsilon^\rho_\xi(A_0^k)=\upsilon^\rho_\xi(A_j^k),
\] 
where the equality follows from the stationarity of $\rho$.

The next lemma establishes an upper bound for the $p_k$'s, which is a power law in $L_k$ with exponent $\alpha$. 

\begin{lemma}
\label{lmult1}
For every $k\in\mathbb{Z}_+$ we have  
\begin{eqnarray}\label{cotlenv} p_k\leq L_k^{-\alpha}.
\end{eqnarray}
\end{lemma}

\begin{proof}
We proceed by induction on $k$. 
Using Markov's inequality for $\rho^\varepsilon$, we get 
\[
p_0=\upsilon^\rho_\xi(A_0^0)\stackrel{\eqref{defZ}}{=}\upsilon^\rho_\xi(Z_0>L_0)\stackrel{\eqref{e:Z_rho}}{=}\mathbb{P}(\rho>L_0)\leq \frac{\mathbb{E}(\rho^\varepsilon)}{L_0^{\varepsilon}},
\]
which, together with \ref{L2}, implies $p_0\leq L_0^{-\alpha}$.
 
Using \eqref{e:|lkj|} and the decoupling inequality \eqref{e:decoup_upsilon} we have
\begin{equation} 
\label{e:contract_relation}
p_{k+1} = \upsilon^\rho_\xi(A_0^{k+1}) \stackrel{\eqref{AA}}{\leq} \sum_{\substack{ i_1, i_2 \in\, l_{k+1,0} \\ \abs{ i_1-i_2 } \geq  2}} \upsilon^\rho_\xi(A_{i_1}^k\cap A^k_{i_2})
\stackrel{\eqref{e:decoup_upsilon}, \eqref{e:|lkj|}}{\leq}  L_k^{2(\gamma-1)}\big[p_k^2+c_1L_k^{-\varepsilon}\big]
\end{equation}
which is a recursive inequality relating $p_{k+1}$ to $p_k$.

Now assume that for some $k\in\mathbb{Z}_+$, $p_k\leq L_k^{-\alpha}$.
Plugging thins bound into \eqref{e:contract_relation}, we get
\begin{equation}
\label{e:induction_pk}
p_{k+1} \leq L_k^{2\gamma-2}(L_k^{-2\alpha}+c_1L_k^{-\varepsilon})
\stackrel{\eqref{defalp}}{\leq} (1+c_1)L_k^{2\gamma-2-2\alpha}
\end{equation}
which implies
\[
 \dfrac{p_{k+1}}{L_{k+1}^{-\alpha}}\leq (1+c_1)L_k^{2\gamma-2-2\alpha}L_{k+1}^\alpha\leq(1+c_1)L_k^{2\gamma-2-2\alpha+\gamma\alpha}\stackrel{\ref{L3}}{\leq}1.
 \]
That is to say that $p_{k+1}$ is also bounded above by $L_{k+1}^{-\alpha}$.
This concludes the proof.\end{proof}

\begin{remark}
\label{r:decoupling_bound}
The crucial assumption imposed on the environment for the previous lemma to hold is that it satisfies the decoupling inequality \eqref{e:decoup_upsilon}.
Although we state our results for a sequence of $i.i.d.$ $\xi$'s, the reader can easily verify that our arguments apply as soon as this decoupling inequality is satisfied and that we can assure that $p_0 \leq L_0^{-\alpha}$ (which, in our case was accomplished by the definition of $L_0$).
\end{remark}

\subsection{Crossings events}
\label{s:multiscale2}

In this section we will study the probability of crossing events within certain rectangles of $\mathbb{Z}_+^2$.
Our goal is to show that, for large enough $p$, rectangles whose bases project onto good blocks are crossed with overwhelming probability.
In the next section we will use this crossings as building blocks to construct an infinite cluster.

Before we state our results, let us introduce the relevant notation.
Let $a, b, c, d \in \mathbb{Z}_+$ with $a<b$ and $c<d$. 
Denote by
\begin{eqnarray}
R=R\big([a,b)\times[c,d)\big)\label{rectangle}
\end{eqnarray}
the subgraph of $\mathbb{Z}^2_+$ whose vertex and edge sets are given repectively by
\begin{eqnarray}
V(R)&=&[a,b]\times[c,d]\textrm{\ \ \ and }\nonumber\\
E(R)&=&\big\{\{(x,y),(x+i,y+1-i)\};\ (x,y)\in[a, b-1]\times[c, d-1], \ i\in\{0,1\}\big\},\nonumber
\end{eqnarray}
where $[a,b]$ denotes the set of all integers between $a$ and $b$, including them both. 
Roughly speaking, $R$ is the rectangle $[a,b]\times[c,d]$ with the edges along the right and top sides removed as shown in Figure \ref{ilustdoret}.
\begin{figure}[htb!]
	\centering
	\begin{tikzpicture}[scale=0.8, every node/.style={scale=0.8}]
	\draw[help lines, step=0.5cm, dotted](-0.2,-0.2) grid (5.2, 3.2);
	\draw[->](0.3,-0.5) -- (5.2,-0.5);
	\draw[->](-0.5,0) -- (-0.5,3.2);
	\draw (1,-0.45) -- (1,-0.55); \node at (1,-0.7) {$a$};
	\draw (4,-0.45) -- (4,-0.55); \node at (4,-0.7) {$b$};
	\draw (-0.45,0.5) -- (-0.55,0.5);\node at (-0.7,0.5) {$c$};
	\draw (-0.45,2.5) -- (-0.55,2.5);\node at (-0.7,2.5) {$d$};
	\foreach \x  in {0,1,2,3,4,5,6} \foreach \y in {0,1,2,3,4} \filldraw (1+0.5*\x,0.5+0.5*\y ) circle(1.5pt);
	\foreach \x  in {0,1,2,3,4,5}\draw[thick] (1+0.5*\x,0.5)--(1+0.5*\x,2.5);
	\foreach \x  in {0,1,2,3}\draw[thick] (1,0.5+0.5*\x )--(4,0.5+0.5*\x );
	\end{tikzpicture}
	\caption{\label{ilustdoret}Illustration of the rectangle $R\big([a,b)\times[c,d)\big)$.}
\end{figure}
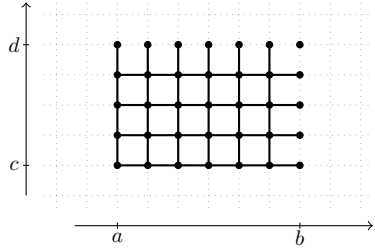

We define horizontal and vertical crossing events in $R$ as follows
\begin{eqnarray} 
\mathcal{C}_h(R)&=&\big\{\{a\}\times[c,d]\leftrightarrow\{b\}\times[c,d]\textrm{ in }R\big\},
\label{crossingh}\\
\mathcal{C}_v(R)&=&\big\{[a,b]\times\{c\}\leftrightarrow[a,b]\times\{d\}\textrm{ in }R\big\},
\label{crossingv}
\end{eqnarray} 
where $\{A\leftrightarrow B \textrm{ in } R\}$ is the event that there are sites $v\in A$ and $w\in B$ that are linked by an open path contained in $R$.

We now specify the rectangles that we will attempt to cross.
Their bases will be blocks at some scale $k$.
Moreover, they will be very elongated on the vertical direction meaning that their heights $H_k$ will be much larger than the length of their bases $L_k$.
More precisely, fix a constant
\begin{eqnarray}
\label{mu}
\mu\in\big(\tfrac{1}{\gamma},1\big)\end{eqnarray} 
and define recursively the sequence $(H_k)_{k\in\mathbb{Z}_+}$ by
\[
H_0=100 \,\,\textrm{ and } \,\, H_k=2\lceil \exp(L_k^\mu)\rceil H_{k-1}, \textrm{ for } k\geq 1.
\]
The choice $H_0=100$ is arbitrary and we could have used any other positive integer.
For $i, j, k\in\mathbb{Z}_+$ denote 
\begin{eqnarray}
\label{defC}
C_{i, j}^k:=\mathcal{C}_h\Big(\big(I^k_i\cup I^k_{i+1}\big)\times\big[j H_k, (j+1)H_k\big)\Big)
\end{eqnarray}
\begin{eqnarray}
\label{defD}
D_{i, j}^k:=\mathcal{C}_v\Big(I^k_i\times\big[j H_k, (j+2)H_k\big)\Big).
\end{eqnarray}
These events are illustrated in Figure \ref{fig:CeD}.

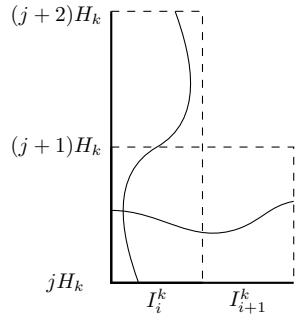
\begin{figure}[htb!]
\centering
\begin{tikzpicture}[scale=0.6, every node/.style={scale=0.8}]
\draw[thick](0,3)--(0,0) -- (4,0);
\draw[dashed](0,3)--(4,3) -- (4,0);
\draw[thick](0,6)--(0,0)--(2,0);
\draw[dashed](0,6)--(2,6) -- (2,0);
\node at (1,-0.4) {$I_i^{k}$};
\node at (3,-0.4) {$I_{i+1}^{k}$};
\node at (-1,0) {$j H_k$};
\node at (-1.2,3) {$(j+1)H_k$};
\node at (-1.2,5.9) {$(j+2)H_k$};
\draw (0,1.6) to [out=0,in=-160] (2.8,1.2) to [out=20,in=190] (4,1.8);
\draw (0.6,0) to [out=110,in=210] (1,3) to [out=30,in=290] (1.4,6);
\end{tikzpicture}
\caption{\label{fig:CeD}Illustration of the events $C_{i, j}^k$ and $D_{i, j}^k$.}
\end{figure}

Let us define, for every $i, j, k \in \mathbb{Z}_+$ and $p\in(0,1)$,
\begin{equation}
\label{eq:def_qk}
q_k(p; i,j):=\max\Bigg\{\underset{\substack{ \Lambda; \ I^k_i \textrm{ and }\\ I^k_{i+1}\textrm{are good}}}{\max}\mathbb{P}^\Lambda_p\big((C_{i,j}^k)^c\big), \underset{\substack{ \Lambda; \ I^k_i \textrm{ is} \\ \textrm{good}}}{\max} \mathbb{P}^\Lambda_p\big((D_{i, j}^k)^c\big) \Bigg\}.
\end{equation}
Translation invariance, allows us to write, for every $k \in \mathbb{Z}_+$, 
\begin{eqnarray}
\label{eq:qk_inv}
q_k(p)\vcentcolon=q_k(p;0,0) = q_k(p; i, j), \textrm{ for any }i, j \in \mathbb{Z}_+.
\end{eqnarray}
There are only finitely many realizations of $\Lambda$ inside the blocks $I^{k}_0$ and $I^{k}_1$ that make these blocks good.
Moreover, for each one of these realizations the probabilities appearing in \eqref{eq:def_qk} vanish as $p$ increases.
Therefore, for any fixed $k$, $q_k(p)\to 0$ as $p\to 1$.

We wish to show that, for sufficiently small $p$, the sequence $q_k(p)$ vanishes fast as $k$ increases.
For that, let us fix a positive constant $\beta$ satisfying
\begin{eqnarray}
\beta \in (\gamma\mu-\gamma+1, 1),
\label{beta}
\end{eqnarray}
which is possible because \eqref{mu} yields $\gamma\mu-\gamma<0$.
We then have:
\begin{lemma}
\label{l3} 
There exist $c_3=c_3(\gamma, L_0, \mu, \beta)\in\mathbb{Z}_+$ and $p=p(\gamma, L_0, \mu, \beta, c_3)$ sufficiently close to 1 such that
\[
q_k(p) \leq \exp(-L_k^{\beta}), \textrm{ for any }k\geq c_3.
\]
\end{lemma}

The proof of this lemma is a straightforward consequence of the two following results.

\begin{lemma}
\label{l4} 
Let $p>1/2$.
There exists $c_4=c_4(\gamma, L_0,\mu, \beta)\in\mathbb{Z}_+$, such that for all $k\geq c_4$ the following holds:
\[
\text{if $q_k(p) \leq \exp({-L_k^\beta})$, then\, $\mathbb{P}^\Lambda_p\big((C_{0,0}^{k+1})^c\big)\leq \exp({-L_{k+1}^\beta})$}
\]
for every environment $\Lambda \in \{\text{$I_0^{k+1}$ is good}\} \cap \{\text{$I_1^{k+1}$ is good}\}$.
\end{lemma}
The assumption $p>1/2$ is not important and was made just for convenience. 

\begin{lemma}
\label{l5}
There exists $c_5=c_5(\gamma, L_0,\mu, \beta)\in\mathbb{Z}_+$, such that for all $k\geq c_5$ the following holds:
\[
\text{if $q_k(p) \leq \exp({-L_k^\beta})$, then\, $\mathbb{P}^\Lambda_p\big((D_{0,0}^{k+1})^c\big)\leq \exp({-L_{k+1}^\beta})$,}
\]
for every environment $\Lambda \in \{\text{$I_0^{k+1}$ is good}\}$.
\end{lemma}

We now show how Lemma \ref{l3} follows from Lemmas \ref{l4} and \ref{l5}.

\begin{proof}[Proof of Lemma \ref{l3}] 
Let $c_3:=\max\{c_4, c_5\}$ and choose $p=p(\gamma, L_0, \mu, \beta, c_3)<1$ such that $q_{c_3}(p)\leq \exp(-L_{c_3}^\beta)$.
Lemmas \ref{l4} and \ref{l5} imply that $q_k(p) \leq \exp(-L_k^{\beta}),$  for any $k\geq c_3$.
\end{proof}

Next we present the proofs of Lemmas \ref{l4} and \ref{l5}. 

\begin{proof}[Proof of Lemma \ref{l4}] 
Fix an environment $\Lambda$ for which $I_0^{k+1}$ and $I_1^{k+1}$ are good blocks. 
Both $I_0^{k+1}$ and $I_1^{k+1}$ may contain at most two bad blocks at scale $k$ in which case they must be adjacent.
Even though it may seem hard to cross these bad blocks, the elongated shape of the rectangles guarantees that there will be many attempts to do so.
Indeed, let us divide the rectangle $R\big([0, 2L_{k+1})\times[ 0, H_{k+1})\big)$ into bands of height $2H_k$, and verify whether crossings take place inside these bands (see Figure \ref{f3}).
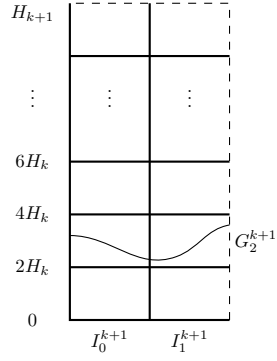
\begin{figure}[htb!]
\centering
\begin{tikzpicture}[scale=0.7, every node/.style={scale=0.7}]
\draw[thick](3,2)--(0,2) -- (0,8);
\draw[dashed](3,2) -- (3,8) -- (0,8);
\draw[thick](1.5,2) -- (1.5,8);
\draw[thick] (0,3) to (3,3);
\draw[thick](0,4) to (3,4);
\draw[thick] (0,5) to (3,5);
\draw[thick](0,7) to (3,7);
\node at (0.75,1.6) {$I_0^{k+1}$};
\node at (2.25,1.6) {$I_1^{k+1}$};
\node at (-0.7,2) {$0$};
\node at (-0.7,3) {$2H_k$};
\node at (-0.7,4) {$4H_k$};
\node at (-0.7,5) {$6H_k$};
\node at (-0.7,6.3) {$\vdots$};
\node at (-0.7,7.8) {$H_{k+1}$};
\node at (3.5,3.5) {$G^{k+1}_2$};
\node at (0.75,6.3) {$\vdots$};
\node at (2.25,6.3) {$\vdots$};
\draw (0,3.6) to [out=0,in=-160] (2,3.2) to [out=20,in=190] (3,3.8);
\end{tikzpicture}
\caption{\label{f3} The occurrence of $G^{k+1}_2$, implies the occurrence of $\{C_{0,0}^{k+1}\}$.}
\end{figure}

For $0\leq j\leq\lceil \exp(L_{k+1}^\mu)\rceil-1$, define the events
\[
G^{k+1}_j:=\mathcal{C}_h\Big(R\big([0, 2L_{k+1})\times[ 2j H_k, (2j+2)H_k)\big)\Big)
\]
and notice that if $C_{0,0}^{k+1}$ does not occur, then none of the events $G^{k+1}_{j}$ can occur. 
Therefore,
\begin{eqnarray}
\label{C00}
\mathbb{P}^\Lambda_p\big((C_{0,0}^{k+1})^c\big)\leq\big(1-\mathbb{P}^\Lambda_p(G^{k+1}_{0})\big)^{\exp(L_{k+1}^\mu)},
\end{eqnarray} 
where we have used the independence and the shift-invariance of the events $G^{k+1}_j$.

In order to obtain a lower bound on the probability of $G_0$ we will build horizontal crossings in $R\big([0, 2L_{k+1})\times[0, 2 H_k)\big)$ using the events $C^k_{i,0}$ and $D^k_{i,0}$ supported in rectangles that project onto good blocks at scale $k$.
However, in case rectangles that project onto bad block of type $I^k_i$ appear, we will try to cross them straight at their bottom.
The strategy is illustrated in Figure~\ref{f4}.

Recall the definition of $l_{{k+1},l}$ in \eqref{firstandlast} and 
denote $j_0$ (resp.\ $j_1$) the earliest index $i \in l_{{k+1},0}$ (resp.\ $i \in l_{{k+1},1}$) such that $I^k_{j_0}$ (resp.\ $I^k_{j_1}$) is a bad block at scale $k$.
In case $I^{k+1}_0$ (resp.\ $I^{k+1}_1$) does not contain any bad block at scale $k$, we define $j_0=0$ (resp. $j_1 = \lfloor L^{\gamma -1}_k\rfloor$) that is, we force $j_0$ (resp. $j_1)$ to be smallest index in $l_{{k+1},0}$ (resp.\ $l_{{k+1},1}$).
Now, for, $l = 0,1$, define
\[
I_l^*:=\big(I^k_{j_l-1}\cup I^k_{j_l}\cup I^k_{j_l+1}\cup I^{k}_{j_l+2}\big)\cap \big(I^{k+1}_0\cup I^{k+1}_1\big).
\]
The intervals $I_0^*$ and $I^*_1$ are just enlarged versions of $I^k_{j_0}$ and $I^k_{j_1}$ that contain the bad blocks at scale $k$ that lie inside $I^{k+1}_0$ and $I^{k+1}_1$, and also the good blocks at scale $k$ that are adjacent to these bad blocks to the left and to the right (as long as these are still contained in $I^{k+1}_0\cup I^{k+1}_1$).
See Figure \ref{f4}.

\begin{figure}[htb!]
	\centering
	\begin{tikzpicture}[scale=0.8, every node/.style={scale=0.8}]
	\draw[thick](7,0)--(0,0)--(0,2);
	\draw[dashed](7,0)--(7,2)--(0,2);
	\draw[thick](7,2)--(7,0)--(14,0);
	\draw[dashed](7,2)--(14,2)--(14,0);
	\foreach \x in {0,1,2,3,4,5,7,8,9,10,11,12}\draw[dotted] (1+\x,0) to (1+\x,2);
	\draw[dotted] (0,1) to (14,1);
	\draw[thick] (4,-1) to (8,-1);
	\draw[thick] (11,-1) to (14,-1);
	\draw[thick] (0,-2) to (14,-2);
	\draw[thick] (0,-1.9) to (0,-2.1) ;
	\draw[thick] (14,-1.9) to (14,-2.1) ;
	\draw[thick] (7,-1.9) to (7,-2.1) ;
	\draw[thick] (11,-0.9) to (11,-1.1) ;
	\draw[thick] (14,-0.9) to (14,-1.1) ;
	\draw[thick] (4,-0.9) to (4,-1.1) ;
	\draw[thick] (8,-0.9) to (8,-1.1) ;
	\node at (-0.5,0) {$0$};
	\node at (-0.5,1) {$H_k$};
	\node at (-0.5,2) {$2H_k$};
	\node at (0.5,-0.5) {$I_0^{k}$};
	\node at (1.5,-0.5) {$I_1^{k}$};
	\node at (3.5,-0.5) {$\cdots$};
	\node at (9,-0.5) {$\cdots$};
	\node at (5.5,-0.5) {$I_{j_0}^{k}$};
	\node at (12.5,-0.5) {$I_{j_1}^{k}$};
	\node at (6,-1.4) {$I_{0}^{*}$};
	\node at (12.5,-1.4) {$I_{1}^{*}$};
	\node at (3.5,-2.4) {$I_{0}^{k+1}$};
	\node at (10.5,-2.4) {$I_{1}^{k+1}$};
	\foreach \x in {0,...,9} \draw (5+0.1*\x,0) -- (5.13+0.1*\x,-0.13);
	\foreach \x in {0,...,19} \draw (12+0.1*\x,0) -- (12.13+0.1*\x,-0.13);
	\foreach \x in {0,1,2,3,4,6,7,8,9,10,11} \draw (0.7+\x,2) to [out=-80,in=20] (0.5+\x,1) to [out=-160,in=100] (0.3+\x,0);
	\foreach \x in {0,1,2,3,6,7,8,9,10} \draw (0+\x,0.7) to [out=10,in=110] (1+\x,0.5) to [out=-70,in=190] (2+\x,0.3);
	\draw (4,0.1)--(8,0.1);
	\draw (11,0.1)--(14,0.1);
	\end{tikzpicture}
	\caption{\label{f4}  
		In this picture, $I^k_{j_0}$, $I^k_{j_1}$ and $I^k_{j_1+1}$ are the only bad blocks at scale $k$ inside $I^{k+1}_0 \cup I^{k+1}_1$. 
		We use the crossings provided by the events $C^k_{i,0}$ and $D^k_{i,0}$ to traverse rectangles that project onto good blocks at scale $k$. 
		Rectangles that project onto bad blocks at scale $k$ and their neighbors are traversed straight at their bottom, yielding the occurrence of $B_0$ and $B_1$. 
		All together, the occurrence of the $C^k_{i,0}$, $D^k_{0,1}$, $B_0$ and $B_1$ implies the occurrence of $G^{k+1}_0$. }
\end{figure}
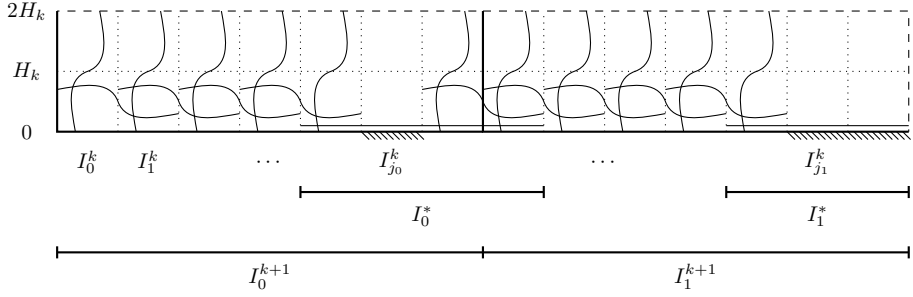

In order to bound below the probability of crossing the possible bad blocks at scale $k$, we will introduce the events $B^k_l$, for $l=0,1$, as follows
\[
B^k_l:=\big\{\textrm{all edges of the form } \{(m,0),(m+1,0)\}\textrm{ with }m\in I_l^* \textrm{ are open}\big\}.
\]
Since we are assuming $p>1/2>1/e$, and $I^*_0$ and $I^*_1$ have length at most $4L_k$, we have 
\begin{eqnarray}\label{B_01}\mathbb{P}^\Lambda_p(B^k_0\cap B^k_1)\geq p^{8 L_k}\geq e^{-8L_k}.
\end{eqnarray}

Notice that
\begin{eqnarray}
\label{G_0} \bigg(\bigcap_{\substack{i; \  I_i^k, I_{i+1}^k\\ \textrm{are good}}}C_{i,0}^k\bigg)\cap\bigg(\bigcap_{\substack{j;\  I_j^k \\ \textrm{is good}}}D_{j,0}^k\bigg)\cap B^k_0\cap B^k_1\subseteq G^{k+1}_0
\end{eqnarray}
where the intersections run over the indices $0\leq i, j\leq 2\lfloor L_k^{\gamma-1}\rfloor-1$. 
Since the events $B^k_l$, $C^k_{i,0}$ and $D^k_{i,0}$ are increasing, it follows from the FKG inequality and \eqref{eq:def_qk} \eqref{eq:qk_inv}, \eqref{B_01}, \eqref{G_0} that
\begin{eqnarray}
\mathbb{P}^\Lambda_p\big(G^{k+1}_0\big)&\geq& \big(1-q_k(p)\big)^{4\lfloor L_k^{\gamma-1}\rfloor}e^{-8L_k}
\geq \big(1-4L_k^{\gamma-1}q_k(p)\big)e^{-8L_k}.
\label{PG_0}
\end{eqnarray}

Pick $c_6=c_6(L_0, \gamma, \beta)$ sufficiently large such that  $4L_k^{\gamma-1}e^{-L_k^\beta}\leq 1/e$ for any $k\geq c_6$.
Now if $q_k(p) \leq \textrm{exp} (-L_k^\beta)$, then \eqref{PG_0} implies
\begin{eqnarray}
\mathbb{P}^\Lambda_p\big(G^{k+1}_0\big)&\geq& \big(1-4L_k^{\gamma-1}e^{-L_k^\beta}\big)e^{-8L_k}
\geq {e^{-8L_k-1}} \geq {e^{-9L_k}}
\label{PG_02}
\end{eqnarray}
for every $k\geq c_6$.
Plugging \eqref{PG_02} into \eqref{C00} and dividing by $\exp(-L_{k+1}^\beta)$ we get
\begin{eqnarray}
\frac{\mathbb{P}^\Lambda_p\big((C_{0,0}^{k+1})^c\big)}{\exp(-L_{k+1}^\beta)}&\leq&\exp(L_{k+1}^\beta)\Big(1-{e^{-9L_k}}\Big)^{\exp(L_{k+1}^\mu)}\nonumber\\
&\leq&\exp(L_{k+1}^\beta) \exp\Big(\exp(-9 L_k+L_{k+1}^\mu)\Big)\nonumber\\
&\stackrel{(\ref{cotaL})}{\leq}&\exp(L_k^{\gamma\beta}-{\exp\big(-9 L_k+(\tfrac{2}{3})^\mu L_k^{\gamma\mu}\big)}),
\label{e:contract_ck00}
\end{eqnarray}
for every $k\geq c_6$.

By (\ref{mu}), $\gamma\mu>1$ therefore, we can take $c_4=c_4(\gamma, L_0, \mu, \beta, c_6)\geq c_6$ sufficiently large such that for every $k\geq c_4$ the right-hand side in \eqref{e:contract_ck00} is at most $1$. 
This finishes the proof.
\end{proof}

\begin{proof}[Proof of  Lemma \ref{l5}] 
Fix an environment $\Lambda \in \big\{\text{$I_0^{k+1}$ is good}\big\}$.  
We will estimate $\mathbb{P}^\Lambda_p\big(D_{0,0}^{k+1}\big)$ using a Peierls-type argument in a rescaled lattice.
Each rectangle $I^k_i\times\big[j H_k, (j+1)H_k\big)$ will correspond to a vertex $(i,j)$ in this rescaled lattice.
Such vertex $(i, j)$ is said open if the event $C_{i,j}^k\cap D_{i,j}^k$ occurs in the original lattice, see Figure \ref{f5}.
Therefore, the rescaled lattice is just $\mathbb{Z}_+^2$ and the resulting process is a dependent percolation on it.

Since $I_0^{k+1}$ is good, either $I^k_i$ is good for every $i\in\big\{0,1, \ldots, \big\lfloor \tfrac{1}{2}{\lfloor L_k^{\gamma-1}\rfloor}\big\rfloor-1\big\}$ or $I^k_i$ is good for every $i\in\big\{\big\lfloor \tfrac{1}{2}{\lfloor L_k^{\gamma-1}\rfloor}\big\rfloor+1, \ldots, \lfloor L_k^{\gamma-1}\rfloor-1\big\}$. 
Assume without loss of generality that the former holds and define
\begin{eqnarray}
\ell_k:=\big\lfloor \tfrac{1}{2}{\lfloor L_k^{\gamma-1}\rfloor}\big\rfloor-1.\label{Ldaprovcruz}
\end{eqnarray} 
Consider the rectangle
\[
R=R\Big(\big[0, \ell_k\big)\times\big[0, 4 \lceil \exp(L_{k+1}^\mu)\rceil\big)\Big),
\] 
and the event $\mathcal{C}_v(R)$ that this rectangle is crossed vertically, as defined in \eqref{crossingh}.

If $\mathcal{C}_v(R)$ does not occur, then there is a sequence of distinct vertices $(i_0,j_0), (i_1,j_1), \ldots, (i_n,j_n)$ in $R$ satisfying 
\begin{enumerate}
\item $\max\big\{\abs{i_l-i_{l-1}},\,\abs{j_l-j_{l-1}}\big\}=1$,
\item $(i_0,j_0)\in\{0\}\times \big[0,4 \lceil \exp(L_{k+1}^\mu)\rceil\big]$ and $(i_n,j_n)\in\{\ell_k\}\times \big[0,4 \lceil \exp(L_{k+1}^\mu)\rceil\big]$,
\item $(i_k,j_k)$ is closed for every $k=0,\ldots, n$.
\end{enumerate}

Note that there are at most $4 \lceil \exp(L_{k+1}^\mu)\rceil8^n$ sequences with $n+1$ vertices that satisfy $1.$ and $2$.
Also, the probability that a fixed vertex in $R$ be open is at least 
\[
1-2q_k(p)\geq1-2\exp(-L_k^\beta).
\] 
The geometry of the crossing events in the original lattice implies that for any $(i,j)\in\mathbb{Z}^2_+$, the event $\{(i,j) \textrm{ is open}\}$ in the renormalized lattice depends on $\{(i',j')\textrm{ is open}\}$ for, at most $7$ distinct vertices $(i',j')$ (see Figure \ref{f5}). 
Therefore, for every set containing $n+1$ vertices, there are at least $\lfloor{n/7}\rfloor$ vertices whose states are mutually independent.

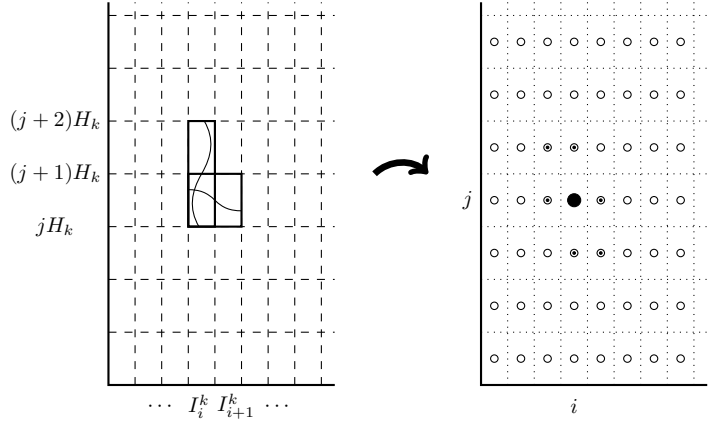
\begin{figure}[htb!]
\centering
\begin{tikzpicture}[scale=0.7, every node/.style={scale=0.8}]
\draw[thick](1,10.25)--(1,3)--(5.25,3);
\foreach \x in {1.5,2,2.5,3,3.5,4,4.5,5}\draw[dashed] (\x,3) to (\x,10.25);
\foreach \x in {3,4,5,6,7,8,9,10}\draw[dashed] (1,\x) to (5.25,\x);
\draw[thick](2.5,6) rectangle (3.5,7);
\draw[thick](2.5,6) rectangle (3,8);
\draw (0.3+2.5,8) to [out=-60,in=60] (0.25+2.5,7) to [out=-120,in=120] (0.2+2.5,6);
\draw (2.5,6+0.7) to [out=0,in=135] (3,6+0.5) to [out=-45,in=180] (3.5,6+0.3);
\node at (2.7,-0.4+3) {$I_i^k$};
\node at (3.4,-0.4+3) {$I_{i+1}^k$};
\node at (2,-0.4+3) {$\cdots$};
\node at (4.2,-0.4+3) {$\cdots$};
\node at (0,6) {$j H_k$};
\node at (0,7) {$(j+1) H_k$};
\node at (0,8) {$(j+2) H_k$};
\draw[->, line width=2.5pt ] (6,7) to[out=45,in=135] (7,7);
\draw[thick] (8,10.25)--(8,3)--(12.25,3);
\foreach \x in {2,2.5,3,3.5,4,4.5,5,5.5}\draw[dotted] (6.5+\x,3) to (6.5+\x,10.25);
\foreach \x in {4,5,6,7,8,9,10}\draw[dotted] (8,\x) to (12.25,\x);
\foreach \x in {1.5,2,2.5,3,3.5,4,4.5,5} \foreach \y in {3,4,5,6,7,8,9} \draw (6.75+\x,0.5+\y) circle (2pt);
\filldraw (2.75+7,5.5) circle (.7pt);
\filldraw (3.25+7,5.5) circle (.7pt);
\filldraw (2.25+7,6.5) circle (.7pt);
\filldraw (2.75+7,6.5) circle (3.5pt);
\filldraw (3.25+7,6.5) circle (.7pt);
\filldraw (2.25+7,7.5) circle (.7pt);
\filldraw (2.75+7,7.5) circle (.7pt);
\node at (9.75,+2.6) {$i$};
\node at (7.75,6.5) {$j$};
\end{tikzpicture}
\caption{\label{f5} On the left, we illustrate the occurrence of the event $C_{i,j}^k\cap D_{i,j}^k$ on the original lattice. 
On the right, we depict the renormalized square lattice where the circles represent the sites. 
The occurrence of $C_{i,j}^k\cap D_{i,j}^k$ in the original lattice implies that the site $(i,j)$ (represented as a black circle) is open in the renormalized lattice. 
The state of the site $(i,j)$ only depends on the state of the other six sites represented as  circles with a dot inside.}
\end{figure}

Therefore,
\begin{eqnarray}
\mathbb{P}
\big(\mathcal{C}_v(R)^c\big)&\leq&\hspace{-0.3cm}\sum_n \mathbb{P}\big(\textrm{there is a sequence of $n+1$ vertices satisfying $1.$, $2.$ and $3.$}\big)\nonumber \\
&\leq&\hspace{-0.3cm}\sum_{n\geq \ell_k}4 \lceil \exp(L_{k+1}^\mu)\rceil8^n\Big(2\exp(-L_k^\beta)\Big)^{\lfloor{n/7}\rfloor}\nonumber\\
&\leq&\hspace{-0.3cm}4 \lceil \exp(L_{k+1}^\mu)\rceil \sum_{n\geq \ell_k}\exp(n\ln 8 + \lfloor{n/7}\rfloor\ln 2 -\lfloor{n/7}\rfloor L_k^\beta)\nonumber\\
&\stackrel{(\ref{Ldaprovcruz})}{\leq}&\hspace{-0.3cm}  c_7\exp(L_{k+1}^\mu-c_8\cdot L_k^{\beta+\gamma-1})\nonumber,
\end{eqnarray}
for some $c_7=c_7(\gamma, L_0, \beta)>0$ and $c_8=c_8(\gamma, L_0, \beta)>0$ sufficiently large.
\medskip

Since $\mathbb{P}^\Lambda_p\big(D_{0,0}^{k+1}\big) \geq \mathbb{P}\big(\mathcal{C}_v(R)\big)$, we have:
\begin{eqnarray}
\frac{\mathbb{P}^\Lambda_p\big((D_{0,0}^{k+1})^c\big)}{\exp(-L_{k+1}^\beta)}\leq c_7\exp(L_{k}^{\gamma\mu}+L_k^{\gamma\beta}-c_8 L_k^{\beta+\gamma-1})\label{lcrosslast}\end{eqnarray}
It follows from the choice of $\beta$ in (\ref{beta}), that $$\beta+\gamma-1>\max\{\gamma\beta,\gamma\mu\}.$$ The proof now follows by choosing $c_5=c_5(\gamma, L_0, \mu, \beta)$ sufficiently large so that the right-hand side of the \eqref{lcrosslast} is less than $1$ whenever $k \geq c_5$.
\end{proof}

\section{Proof of Theorems \ref{t3_1} and \ref{t3_2}}

In this section we put together the results obtained in Section \ref{s:multiscale} in order to prove Theorem \ref{t3_1}.
We also present the proof of Theorem \ref{t3_2}.

\subsection{Proof of Theorem \ref{t3_1}}
\label{demt3_1}

We split the proof of Theorem \ref{t3_1} into two parts.
In the first one, we show that for a given large $p$, $\{o \leftrightarrow \infty\}$ occurs with positive probability for almost all realizations of $\Lambda$.
In the second part, we show that for small enough $p$ this event occurs with null probability.
By standard coupling, $\mathbb{P}^{\Lambda}_p$ is stochastically increasing in $p$ therefore, our results show the existence of a non-trivial critical threshold $p_c(\Lambda)$, for almost all $\Lambda$.
Nevertheless by standard ergodicity arguments, we have that $p_c$ does not depend on $\Lambda$.

\begin{proof}[Proof of Theorem \ref{t3_1} (Percolation for large $p$)] 

Let $\xi$ be any positive random variable such that $\mathbb{E}(\xi^\eta)<\infty$ for a given $\eta>1$.
Denote 
\[
m:=\gcd\big\{k\in\mathbb{Z}^*_+; \ \mathbb{P}(\lceil\xi\rceil=k)\not=0\big\}.
\]
Now define the positive integer-valued, aperiodic random variable $\xi'=\lceil \xi\rceil/m$ which also satifies $\mathbb{E}\big((\xi')^\eta\big)<\infty$. 
If for some $p<1$, we have $\mathbb{P}^{\Lambda'}_p(o\leftrightarrow\infty)>0$ for $\upsilon_{\xi'}$-a.e.\ environment $\Lambda'$, then also $\mathbb{P}^{\Lambda}_{p^{1/m}}(o\leftrightarrow \infty)>0$ for $\upsilon_{\xi}$-a.e.\ $\Lambda$ as it can be seem by a simple coupling argument.
In view of this, we will assume henceforth that $\xi$ is a positive integer-valued, aperiodic random variable.
In particular, we can apply Lemmas \ref{lmult1} and \ref{l3} from Section \ref{s:multiscale}.

Lemma \ref{lmult1} implies that
\[
\upsilon^\rho_\xi\bigg(\hspace{-1cm}\bigcup_{\hspace{1cm}0\leq i\leq\lfloor L_k^{\gamma-1}\rfloor-1}\hspace{-1 cm}\big\{I_i^k\textrm{ is bad}\big\}\bigg)\leq L_k^{\gamma-1}L_k^{-\alpha}\overset{(\ref{defalp})}{\leq} L_k^{-\frac{\alpha}{2}}.
\]
By \eqref{cotaL}, $\sum_{k}L_k^{-\alpha/2}<\infty$, therefore, it follows from  Borel-Cantelli's Lemma, that for $\upsilon^\rho_\xi$-almost every environment $\Lambda$, there exists $c_9=c_9(\Lambda, c_3)>c_3$ such that for every $k\geq c_9$, all the blocks at scale $k$ inside $I^{k+1}_0$ are good. 
Fix such an environment $\Lambda$.

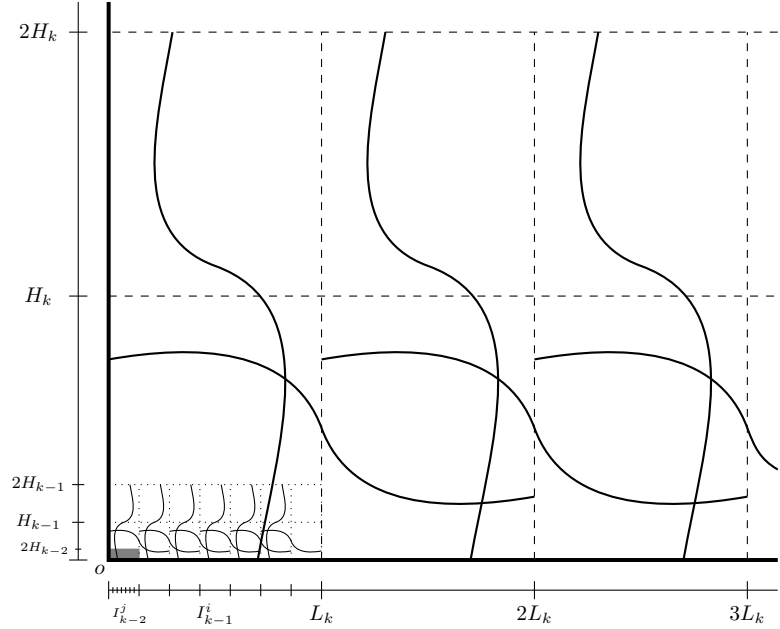
\begin{figure}[htb!]
\centering
\begin{tikzpicture}[scale=0.4, every node/.style={scale=0.8}]
\filldraw[black!50!] (0,0)rectangle(1,0.36);
\draw[line width=1.5pt](22,0)--(0,0)--(0,18.5);
\draw[dashed](22,17.5)--(0,17.5);
\draw[dashed] (7,0)--(7,17.5);
\draw[dashed] (21,0)--(21,17.5);
\draw[dashed] (14,0)--(14,17.5);
\draw[dashed] (0,8.75)--(22,8.75);
\draw[dotted] (0,2.5)--(7,2.5);
\draw[dotted] (0,1.25)--(7,1.25);
\foreach \x in {0,1,...,6}\draw[dotted] (1+\x,0) to (1+\x,2.5);
\foreach \x in {0,1,2,3,4,5} \draw (0.7+\x,2.5) to [out=-80,in=20] (0.5+\x,1.25) to [out=-160,in=100] (0.3+\x,0);
\foreach \x in {0,1,2,3,4,5} \draw (0+\x,0.95) to [out=10,in=110] (1+\x,0.625) to [out=-70,in=190] (2+\x,0.3);
\draw[thick]  (2.1,17.5)to [out=-180+80,in=-180-20] (3.5,9.75) to [out=-180+160,in=-180-100] (4.9,0);
\draw[thick]  (2.1+7,17.5)to [out=-180+80,in=-180-20] (3.5+7,9.75) to [out=-180+160,in=-180-100] (4.9+7,0);
\draw[thick]  (2.1+14,17.5)to [out=-180+80,in=-180-20] (3.5+14,9.75) to [out=-180+160,in=-180-100] (4.9+14,0);
\draw[thick]  (0,6.65) to [out=10,in=110] (7,4.375) to [out=-70,in=190] (14,2.1);
\draw[thick]  (7,6.65) to [out=10,in=110] (14,4.375) to [out=-70,in=190] (21,2.1);
\draw[thick]  (14,6.65) to [out=10,in=110] (21,4.375)to [out=-70,in=150] (22,3);
\draw (0,-1)--(22,-1);
\foreach \x in{0,1,2,3}\draw (\x*7,-1.3)--(\x*7,-0.7);
\foreach \x in{1,2,3,4,5,6}\draw (\x,-1.2)--(\x,-0.8);
\foreach \x in{1,2,3,4,5,6}\draw (\x*0.14,-1.1)--(\x*0.14,-0.9);
\foreach \x/\y in{1/,2/2,3/3}\node at (\x*7,-1.8){$\y L_k$};
\node at (3.5,-1.8){\footnotesize $I_{k-1}^i$};
\node at (0.7,-1.8){\tiny  $I_{k-2}^j$};
\draw (-1,0)--(-1,18.5);
\draw (-1.3,0)--(-0.7,0);
\draw (-1.3,8.75)--(-0.7,8.75);
\draw (-1.3,17.5)--(-0.7,17.5);
\draw (-1.2,2.5)--(-0.8,2.5);
\draw (-1.2,1.25)--(-0.8,1.25);
\draw (-1.1,0.36)--(-0.9,0.36);
\node at (-2.3,8.75){$H_k$};
\node at (-2.3,17.5){$2H_k$};
\node at (-2.3,1.25){\footnotesize$H_{k-1}$};
\node at (-2.3,2.5){\footnotesize$2H_{k-1}$};
\node at (-2.1,0.36){\tiny $2H_{k-2}$};
\node at (-0.3,-0.3){$o$};
\end{tikzpicture}
\caption{The simultaneous occurrence of the events $C^k_{i,0}$ e $D^k_{i,0}$ for $k \geq c_{10}$ and $i=0,\cdots, \lfloor L_k^{\gamma-1} \rfloor -2$, implies the existence of an infinite cluster.}
\label{ultfigu}
\end{figure}

Let $c_{10} =c_{10}(\Lambda) \geq c_9$ be any integer such that
\begin{equation}
\label{e:bound_sum_l_k}
\sum_{k\geq c_{10}}2 L_k^{\gamma-1} \exp(-L_k^\beta) < 1/2.
\end{equation}

Recall the definition of $C_{i,0}^k$ and $D_{i,0}^k$ in (\ref{defC}) and (\ref{defD}) and note that 
\begin{equation}
\label{e:several_crossings}
\bigcap_{k\geq c_{10}}\bigg(\bigcap_{i=0}^{\lfloor L_k^{\gamma-1}\rfloor-2}\hspace{-0.3cm}\big(C_{i,0}^k\cap D_{i,0}^k\big)\bigg)\subseteq \big\{\textrm{there is an infinite cluster}\big\},
\end{equation}
(see Figure \ref{ultfigu}). 
By Lemma \ref{l3}, we can take $p$ sufficiently close to 1, without depending on $\Lambda$, so that $q_k(p) \leq \exp(-L_k^\beta)$ for every $k\geq c_{10}$.
Using the FKG inequality, \eqref{eq:qk_inv} and \eqref{e:bound_sum_l_k} we have
\begin{eqnarray}
\mathbb{P}^\Lambda_p\Bigg(\bigcap_{k\geq c_{10}}\hspace{-0.2cm} \bigg(\bigcap_{i=0}^{\lfloor L_k^{\gamma-1}\rfloor-2}\hspace{-0.3cm} \big(C_{i,0}^k\cap D_{i,0}^k\big)\bigg) \Bigg)&\geq&\prod_{k\geq c_{10}}(1-2q_k(p))^{\lfloor L_k^{\gamma-1}\rfloor-1}\nonumber\\
&\geq&1-\sum_{k\geq c_{10}}2 L_k^{\gamma-1} q_k(p) \nonumber\\
&\geq&1-\sum_{k\geq c_{10}}2 L_k^{\gamma-1} \exp(-L_k^\beta) \geq \frac{1}{2}.
\label{fafafa98}
\end{eqnarray}
From \eqref{e:several_crossings} we get
\[
\mathbb{P}^\Lambda_p\big(\text{there is an infinite cluster}\big) \geq \frac{1}{2},
\]
hence the FKG inequality guarantees that $\mathbb{P}^{\Lambda}_{p}(o \leftrightarrow \infty) >0$.
\end{proof}

Before we present the second part of the proof, let us provide a rough outline of our argument.
First, for a fixed an environment $\Lambda$, we will consider a new lattice, obtained by joining the columns that are close to one another and also by truncating at a maximum distance between the other columns. 
This is done in such a way that, for this new lattice the probability of $\{o\leftrightarrow \infty\}$ is greater than in the original one. 
To bound this probability, we will use a duality argument, noting that the dual model is distributed as the original model which will allow us to apply the bounds obtained in Lemmas \ref{lmult1} and \ref{l3}.

\begin{proof}[Proof of Theorem \ref{t3_1} (Absence of percolation for small $p$)]
Let us fix $0<\kappa\leq 1$ such that $\mathbb{P}(\xi\geq\kappa)\geq 1/2$, which exists because $\xi$ is a positive random variable. 
For a fixed realization of $\Lambda$, or equivalently of $(\xi_k)_{k\in\mathbb{Z}_+^*}$ let us define $J_0=0$ and for every $k\geq 1$,
\[
J_k:=\min\Big\{n;\, \sum_{i=0}^{n} \mathbf{1}_{\{\xi_i \geq \kappa\}} \geq k\Big\}.
\]
We now define another sequence $(\zeta_k)_{k\in\mathbb{Z}_+}$ by 
\[
\zeta_0=0\,\,\,\,\textrm{ and }\,\,\,\, \zeta_k = J_k - J_{k-1}\,\,\,\, \text{ for $k\geq 1$}.
\]
Roughly speaking, $\zeta_k-1$ counts the number of elements in the sequence $\{\xi_j\}_{j}$ that appear between the $(k-1)$-th and the $k$-th appearance of a $\xi_j$ whose value exceeds $\kappa$.
By our choice of $\kappa$, the random variables $\{\zeta_i\}_{i\in\mathbb{Z}_+^*}$ are dominated by independent random variables with geometric distribution with parameter $1/2$ hence, in particular
\begin{equation}
\label{e:moment_zeta}
    \mathbb{E}(\zeta_i^2) < \infty.
\end{equation}

Recall that, for a fixed $\Lambda$, $\mathbb{P}_p^{\Lambda}$ stands for the law of a bond percolation process in $\mathbb{Z}^2_+$ as given by \eqref{ipmsl}.
We now perform three operations that will only increase the probability of existence of an infinite cluster:
\begin{enumerate}
    \item replace $\xi_i$ by $\kappa$ whenever $\xi_i \geq \kappa$;
    \item replace the parameter $p$ in vertical edges by $p^\kappa$;
    \item contract each horizontal edge $\{(i,j),(i+1,j)\}$ for which $\xi_i < \kappa$ into a single site.
\end{enumerate}
These operations lead to a model in $\mathbb{Z}^2_+$ in which each edge is independently declared open with probability
\begin{equation}
\label{e:weigths_enhanced}
p_e=\left\{\begin{array}{ccc}
1-(1-p^\kappa)^{\zeta_i},& \textrm{ if }e=\{(i,j),(i,j+1)\}\\
p^\kappa,& \textrm{ if }e=\{(i,j),(i+1,j)\}
\end{array}\right.,
\end{equation}
as illustrated in Figure \ref{f:weigths_enhanced}.
The proof is finished once we show that, for this model, the probability that $\{o\leftrightarrow \infty\}$ equals $0$.

\begin{figure}[htb!]
	\centering
	\begin{tikzpicture}[]scale=0.4, every node/.style={scale=0.4}
\foreach \x in {0,1,1.3,1.5,2.7,2.9,3.7,3.8,4,5} \draw[] (\x,0)--(\x,3.6);
\foreach \y in {0,0.8,...,4} \draw[] (0,\y)--(5.9,\y);
\foreach \y in {0,0.8,...,4} \draw[] (7.5,\y)--(7.5+3.3,\y);
\foreach \x in {0,0.7,...,3.3} \draw[] (7.5+\x,0)--(7.5+\x,3.6);

\foreach \y in {0,0.8,...,3.2}\draw (8.2,\y) to [out=120, in=240] (8.2,\y+0.8);
\draw[] (8.2,3.2) to [out=120, in=-90] (8.1,3.6);
\foreach \y in {0,0.8,...,3.2} \draw (8.2,\y) to [out=60, in=-60] (8.2,\y+0.8);
\draw[] (8.2,3.2) to [out=60, in=-90] (8.3,3.6);

\foreach \y in {0,0.8,...,3.2} \draw (8.2+0.7,\y) to [out=60, in=-60] (8.2+0.7,\y+0.8);
\draw[] (8.2+0.7,3.2) to [out=60, in=-90] (8.3+0.7,3.6);

\foreach \y in {0,0.8,...,3.2} \draw (8.2+1.4,\y) to [out=120, in=240] (8.2+1.4,\y+0.8);
\draw[] (8.2+1.4,3.2) to [out=120, in=-90] (8.1+1.4,3.6);
\foreach \y in {0,0.8,...,3.2} \draw (8.2+1.4,\y) to [out=60, in=-60] (8.2+1.4,\y+0.8);
\draw[] (8.2+1.4,3.2) to [out=60, in=-90] (8.3+1.4,3.6);
\node at (7.1,-0.3) {$\zeta:$};
\node at (7.5,-0.3) {$1$};
\node at (8.2,-0.3) {$3$};
\node at (8.9,-0.3) {$2$};
\node at (9.6,-0.3) {$3$};
\node at (10.3,-0.3) {$1$};

\node at (0.5,-0.5) {$\xi_{1}$};
\node at (2.1,-0.5) {$\xi_{4}$};
\node at (3.3,-0.5) {$\xi_{6}$};
\node at (4.5,-0.5) {$\xi_{9}$};

	\def\numbers{{0,1,1.3,1.5,2.7,2.9,3.7,3.8,4,5}}
	   \foreach \i in {0,...,8}{
	    \pgfmathsetmacro{\n}{\numbers[\i]}
	     \pgfmathsetmacro{\m}{\numbers[\i+1]}
	      \pgfmathsetmacro{\k}{\i+1}
	     \draw[] (\n,-.2) -- (\m,-.2);
	     \draw[](\n,-.15) -- (\n,-.25);
	     \draw[](5,-.15) -- (5,-.25);
	     	     }
	\end{tikzpicture}
	\caption{On the left-hand side we represent $\mathcal{L}_{\Lambda}$ with the corresponding variables $\xi_i$ whose values are larger than $\kappa$, the ones whose values are smaller then $\kappa$ have been omitted. 
	On the right-hand site we represent the lattice obtained from the $\mathbb{Z}^2_+$ by contracting the edges whose variable $\xi_i$ are smaller than $\kappa$ into a single site and the corresponding varibales $\zeta_i$. 
	Homogeneous percolation on this lattice with parameter $p^{\kappa}$ is equivalent to percolation on $\mathbb{Z}^2_+$ with parameters given by \eqref{e:weigths_enhanced}.}
	\label{f:weigths_enhanced}.
\end{figure}
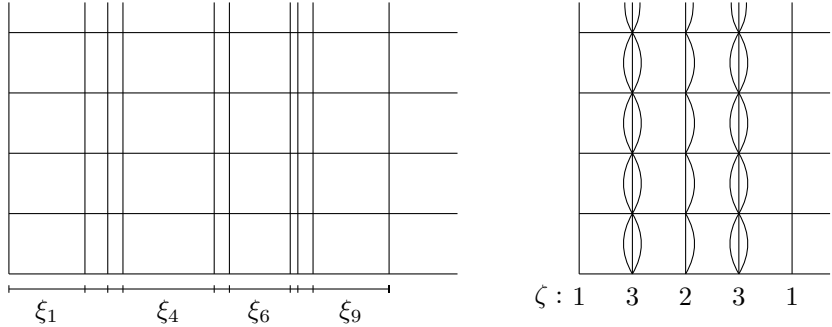

We use a standard duality argument.
Let us now consider the lattice $(\mathbb{Z}^2_+)^\star:=\mathbb{Z}^2_+ - (1/2,1/2)$, which is just a copy of the original $\mathbb{Z}^2_+$ lattice shifted $1/2$ units left and down.
Strictly speaking it is not exactly the dual of $\mathbb{Z}^2_+$ because the edges along its leftmost vertical semiaxis $\{-1/2\}\times (\mathbb{Z}_+ - 1/2)$ and lowermost horizontal semiaxis  $(\mathbb{Z}_+-1/2) \times \{-1/2\}$ do not intercept any edge of $\mathbb{Z}^2_+$.

For each edge $e^\star$ in $(\mathbb{Z}^2_+)^\star$ that actually intersects an edge in $e \in \mathbb{Z}^2_+$ we declare it open (resp.\ closed) if $e$ is closed (resp.\ open).
We still need to specify the state of the edges that lie on the leftmost and lowermost semiaxis of $(\mathbb{Z}^2_+)^\star$.
This is done independently as folows:
Declare the edges in vertical semiaxis $\{-1/2\}\times (\mathbb{Z}_+ - 1/2)$ open with probability $1-p^{\kappa}$.
Also, declare the edges of the type $\{(i-1/2,-1/2),(i+1/2,-1/2)\}$ that lie in the horizontal axis $(\mathbb{Z}_+-1/2) \times \{-1/2\}$ open with probability $(1-p^{\kappa})^{\zeta_i}$. 
It is straightforward to check that the egdes $e^\star$ are open independently with probability
\begin{equation}
\label{e:ipmsl_dual}
p_{e^\star}=
\left\{\begin{array}{ccc}
1-p^\kappa,& \textrm{ if }e^\star=\big\{(i-\tfrac{1}{2},j-\tfrac{1}{2}),(i-\tfrac{1}{2},j+\tfrac{1}{2})\big\}\vspace{0.1cm}\\
(1-p^\kappa)^{\zeta_i},& \textrm{ if }e^\star=\big\{(i-\tfrac{1}{2},j-\tfrac{1}{2}),(i+\tfrac{1}{2},j-\tfrac{1}{2})\big\}
\end{array}\right..
\end{equation}
Comparing with \eqref{ipmsl}, the resulting law in $(\mathbb{Z}^2_+)^\star$ of a process given by \eqref{e:ipmsl_dual} is simply (with an excusable abuse of notation) $\mathbb{P}^\Xi_{1-p^\kappa}$ where, similar to the definition of $\Lambda$ in \eqref{e:def_lambda}, $\Xi$ is the corresponding sequence whose increments are given by the $\zeta_i$'s, that is,
\begin{equation}
\label{e:def_xi}
\Xi:=\big\{x_k\in \mathbb{R};\,  x_0=-\tfrac{1}{2} \textrm{ and }  x_k=x_{k-1}+\zeta_k\textrm{ for }k\in\mathbb{Z}_+^* \big\}.
\end{equation}

Now the fact that the $\zeta_i$'s are positive, integer-valued random variables and that the moment condition \eqref{e:moment_zeta} holds enable us to find the corresponding constants $\varepsilon$, $\alpha$, $\gamma$ and $\mu$ along with the corresponding scale sequences $L_k(\zeta_1, \varepsilon, \alpha, \gamma)$ and $H_k(L_0, \gamma, \mu)$.
Having $L_k$ and $H_k$ defined, we can also define the crossing events $C^{k,\star}_{i,1}$ and $D^{k,\star}_{i,1}$ defined analogously as in \eqref{defC} and \eqref{defD}, only that they take place in $(\mathbb{Z}^2_+)^\star$ instead of $\mathbb{Z}^2_+$.
So we can now use Lemmas \ref{lmult1} and \ref{l3} to obtain that, if $p$ is sufficiently small and consequently $1-p^\kappa$ is sufficiently large,
\begin{equation}
\label{e:dual_cross_io}
\mathbb{P}^{\Xi}_{1-p^\kappa}\big( C^{k,\star}_{i,1} \cap D^{k,\star}_{i,1} \big) \geq 1-2q_{k}(1-p^\kappa) \geq 1-2\exp(-L_k^{\beta}) \to 1,
\end{equation}
as $k$ goes to infinity.
It follows that, almost surely, the origin $o$ in $\mathbb{Z}^2_+$ is encompassed by an open dual semicircuit linking the leftmost vertical and the lowermost horizontal semiaxis of $(\mathbb{Z}^2_+)^\star$.
This semicircuit acts as a blocking structure for $o$ to percolate in the original lattice.
Therefore, $\{o \leftrightarrow \infty\}$ has null probability for the model given by \eqref{e:weigths_enhanced}.
This concludes the proof.
\end{proof}

\subsection{Proof of Theorem  \ref{t3_2}}\label{demt3_2}
We now present a proof of Theorem \ref{t3_2}.
Rather than using the multiscale developed in Section \ref{s:multiscale} we  explore the fact that consecutive columns will be typically very far apart to apply a Borel-Cantelli argument showing that long crossings cannot occur, hence no infinite cluster can exist.

\begin{proof}[Proof of Theorem  \ref{t3_2}]

Let $\eta<1$ be such that $\mathbb{E}(\xi^\eta)=\infty$ and set $\epsilon>0$ such that $\eta^{-1}=1+2\epsilon$.
Since $\{\xi_i\}_{i\in\mathbb{Z}_+^*}$ are independent copies of $\xi$, we have
\begin{eqnarray}
\sum_{i=0}\mathbb{P}(\xi_i>i^{1+2\epsilon}) 
=\sum_{i=0} \mathbb{P}(\xi>i^{1+2\epsilon})
= \sum_{i=0} \mathbb{P}( \xi^\eta>n)=\infty.
\label{t3_2.1}
\end{eqnarray}
Therefore, by Borel-Cantelli, $\upsilon_\xi (\xi_i>i^{1+2\epsilon} \text{ i.o.})=1$. 

Let us now fix $\Lambda\in\{\xi_i>i^{1+2\epsilon} \text{ i.o.}\}$ and $p \in (0,1)$.
It is sufficient to show that $\mathbb{P}^\Lambda_p(o\leftrightarrow \infty)=0$.  
Pick an increasing subsequence $i_k=i_k(\Lambda), k\in\mathbb{Z}_+$ such that $\xi_{i_k}>{i_k}^{1+2\epsilon}$ for every $k$.
Roughly speaking, as we will see below, the vertical columns of $\mathcal{L}_\Lambda$ that project to $x_{i_{k-1}}$ and $x_{i_{k}}$ are too distant from each other to allow for the existence of paths that connect between them.
We choose however, to use the equivalent formulation of the percolation model on $\mathbb{Z}^2_+$ with parameters $p_e$ as given in \eqref{ipmsl}.

Recall the notation for horizontal and vertical crossings events in rectangles introduced in (\ref{rectangle}), (\ref{crossingh}) and (\ref{crossingv}).
For each $k\in\mathbb{Z}_+$ let
\[
R_k:=R\Big(\big[0, {i_k}\big)\times \big[0, \lceil \exp(i_k^{1+\epsilon})\rceil\big)\Big),
\]
and notice that
\begin{eqnarray}\mathbb{P}^\Lambda_p(o\leftrightarrow\infty)\leq \mathbb{P}^\Lambda_p\big(\mathcal{C}_h(R_k)\big)+\mathbb{P}^\Lambda_p\big(\mathcal{C}_v(R_k)\big).\label{t2.3}\end{eqnarray}
The probability of $\mathcal{C}_h(R_k)$ is bounded above by the probability that there is an open edge between the columns $\{{i_k}-1\}\times \mathbb{Z}_+$ and  $\{{i_k}\}\times \mathbb{Z}_+$. 
Since $\xi_{i_k}$ is large, the height of $R_k$ is not enough to ensure that such an edge exists with good probability.
In fact, let
\[
J_k=\big\{0, 1, \ldots, \lceil \exp(i_k^{1+\epsilon})\rceil-1\big\}
\]
and note that
\begin{eqnarray}
\mathbb{P}^\Lambda_p\big(\mathcal{C}_h(R_k)\big)
&\leq&\mathbb{P}^\Lambda_p\bigg(\bigcup_{j\in J_k}\big\{\{({i_k-1},j), ({i_k}, j)\} \textrm{ is open}\big\}\bigg) \nonumber\\
&\leq& \lceil \exp(i_k^{1+\epsilon})\rceil p^{\xi_{i_k}}\nonumber\\ 
&\leq& \lceil\exp(i_k^{1+\epsilon})\rceil \exp(i_k^{1+2\epsilon} \ln p)\stackrel{k\to \infty}{\longrightarrow} 0. \label{t2.4}
\end{eqnarray}

In order to bound the probability of $\mathcal{C}_v(R_k)$, we note that, on this event there must be at least one vertical edge connecting the $j$-th and $(j+1)$-th row in $R_k$, for every $j\in J_k$. 
This time, the height of $R_k$ is too large for these edges to exist with good probability.
In fact,
\begin{eqnarray}
\mathbb{P}^\Lambda_p\big(\mathcal{C}_v(R_k)\big)&\leq&\mathbb{P}^\Lambda_p\bigg(\bigcap_{j\in J_k}\bigcup_{l=0}^{i_k-1}\big\{\{({l}, j),({l}, j+1)\} \textrm{ is open}\big\}\bigg) \nonumber\\
&=&\Big(1-\big(1-p\big)^{i_k}\Big)^{\abs{J_k}}\nonumber\\
&=&\big(1-\exp(i_k \ln(1-p))\big)^{\abs{J_k}}\nonumber\\
&\leq& \exp\big(-\abs{J_k}\exp(i_k \ln(1-p))\big) \nonumber\\
&\leq& \exp\big( -\exp(i_k\ln(1-p)+i_k^{1+\epsilon})\big)\stackrel{k\to \infty}{\longrightarrow} 0. \label{t2.5}
\end{eqnarray}

Combining (\ref{t2.3}), (\ref{t2.4}) and (\ref{t2.5}), we get 
$\mathbb{P}^\Lambda_p\big(o\leftrightarrow\infty\big)=0,$ which concludes the proof.
\end{proof}

\end{document}